\documentclass[abstract=true]{scrartcl}

\usepackage[english]{babel}
\usepackage[utf8]{inputenc}
\usepackage{mathtools}
\usepackage{mathpazo}
\usepackage{amsmath,amssymb,amsthm}
\usepackage{graphicx}
\usepackage{subcaption}
\usepackage[margin=1in]{geometry}
\usepackage[hyphens]{url}
\usepackage[hidelinks]{hyperref}
\usepackage{cleveref}
\usepackage{mathrsfs}
\usepackage[normalem]{ulem}
\usepackage{todonotes}
\usepackage{comment}
\usepackage[most]{tcolorbox}
\usepackage{algorithm}

\usepackage{tikz}
\usetikzlibrary{arrows.meta,calc}
\definecolor{newblue}{RGB}{52,101,164}
\definecolor{neworange}{RGB}{230,97,1}
\definecolor{oldgray}{RGB}{130,130,130}

\usepackage{enumitem}

\usepackage{aligned-overset, nicefrac, empheq}
\usepackage{thm-restate, amsfonts, xspace}
\usepackage{bm}
\usepackage{booktabs, multirow}

\newtheorem{theorem}{Theorem}[section]
\newtheorem{lemma}[theorem]{Lemma}
\newtheorem{definition}[theorem]{Definition}

\newtheorem{example}{Example}[section]

\newtheorem{proposition}[theorem]{Proposition}
\newtheorem{corollary}[theorem]{Corollary}
\newtheorem{observation}[theorem]{Observation}

\crefdefaultlabelformat{#2\textup{#1}#3}
\Crefname{assumption}{Assumption}{Assumptions}
\Crefrangeformat{assumption}{Assumptions #3#1#4--#5#2#6}

\def\defi{\vcentcolon=}

\newcommand{\ignore}[1]{}

\title{Optimal Spectral Design with Prior Information\footnote{%
\textbf{Funding:}
This material is based upon work supported by the National Science Foundation under
Awards No.\ DMS-2410944, AF-2504994, 2106444, CMMI-2246414,
and the Office of Naval Research N00014-24-1-2066.
}}
\author{Anton J.\ Kleywegt\thanks{H.\ Milton Stewart School of Industrial and Systems Engineering, Georgia Institute of Technology, Atlanta, Georgia 30332 (\texttt{anton@isye.gatech.edu}, \texttt{johannes.milz@isye.gatech.edu}, \texttt{mohit.singh@isye.gatech.edu}, \texttt{wxie@gatech.edu}).}
\and Johannes Milz\footnotemark[2]
\and Mohit Singh\footnotemark[2]
\and Weijun Xie\footnotemark[2]}
\date{\today}
\allowdisplaybreaks
\begin{document}

\maketitle

\begin{abstract}
We study a class of spectral design problems in which a prior positive semidefinite information matrix is updated by a sum of rank-one matrices constructed from chosen design vectors subject to a bound on their Euclidean norm. The objective of a spectral design problem is any symmetric convex function of the eigenvalues of the updated information matrix. This framework unifies classical optimal experimental design criteria, including A-, D-, and E-optimality.
It also arises in model-based derivative-free optimization, where sampling directions determine the conditioning and accuracy of regression models.
Although the objective is symmetric and convex in the eigenvalues, the optimization problem with design vectors/matrix as decision variables is nonconvex, and optimal solutions of their convex relaxations may not be feasible for the spectral design problem.
We use tight eigenvalue relaxations to obtain a convex reformulation, and we apply the Schur--Horn theorem to construct a simple polynomial-time algorithm for solving the spectral design problem.
We illustrate the optimal spectral designs computed by our algorithm.
Moreover, a small set of numerical experiments shows the potential of spectral designs for derivative-free optimization.
\end{abstract}

{\small 
\noindent\textbf{Key words.}
spectral design, optimal experimental design, information matrix,
Schur--Horn theorem, derivative-free optimization.

\noindent\textbf{AMS subject classifications.}
Primary: 62K05. Secondary: 90C25, 90C56, 15A18, 65K05.
}

\section{Introduction}
\label{sec:introduction}

Experimental design is the selection of conditions under which experiments are conducted. We consider experimental design in a setting where an experimenter can control $d$~real-valued factors, and $k$~combinations of values for the $d$~factors are to be chosen at which experimental observations will be obtained. That is, the experimental designer must select design vectors $\bm{x}^{1}, \ldots, \bm{x}^{k} \in \mathbb{R}^{d}$. The chosen design vectors must be close to a given central point, but there are no other restrictions on the allowable combinations of factor values for the design vectors.
Specifically, we require that $\bm{x}^{1},\ldots,\bm{x}^{k} \in B(\bm{0},1) \defi \left\{\bm{x} \in \mathbb{R}^{d} \, : \, \bm{x}^{\top} \bm{x} \le 1\right\}$.
The requirements that the central point is $\bm{0}$ and that the maximum distance between a chosen design vector and the central point is $1$ are without loss of generality, but the use of Euclidean distance is restrictive.

The best design depends on the purposes of the experiment.
Next we review some widely used criteria and associated optimization problems for the selection of a design.
The first criterion, called A-optimal design, chooses a design that solves
\[
\min\left\{\operatorname{Tr}\left(\left(\sum_{i=1}^{k} \bm{x}^{i} \left(\bm{x}^{i}\right)^{\top}\right)^{-1}\right) \; : \; \bm{x}^{1},\ldots,\bm{x}^{k} \in B(\bm{0},1)\right\}.
\]
where $\operatorname{Tr}(\bm{M})$ denotes the trace of matrix $\bm{M} \in \mathbb{R}^{d \times d}$.
Let $\bm{X} = [\bm{x}^{1},\ldots,\bm{x}^{k}] \in \mathbb{R}^{d \times k}$ denote the design matrix; we call $\bm{X} \bm{X}^{\top} \in \mathbb{R}^{d \times d}$ the information matrix.
For any $\bm{M} \in \mathbb{R}^{d \times d}$, let $\bm{\lambda}(\bm{M}) \defi (\lambda_{1}(\bm{M}),\ldots,\lambda_{d}(\bm{M}))$ denote the eigenvalues of $\bm{M}$.
We will assume that the eigenvalues of a symmetric matrix $\bm{M} \in \mathbb{R}^{d \times d}$ are sorted in non-decreasing order, that is, $\lambda_{1}(\bm{M}) \le \cdots \le \lambda_{d}(\bm{M})$.
Then the optimization problem for A-optimal design can be written as
\[
\min\left\{\sum_{i=1}^{d} \lambda_{i}\left(\bm{X} \bm{X}^{\top}\right)^{-1} \; : \; \bm{x}^{1},\ldots,\bm{x}^{k} \in B(\bm{0},1)\right\}.
\]
Note that $f(\lambda_{1},\ldots,\lambda_{d}) = \sum_{i=1}^{d} \lambda_{i}^{-1}$ is a symmetric convex function on $(0,\infty)^{d}$.

The second criterion, called D-optimal design, chooses a design that solves
\begin{align*}
& \min\left\{-\log\left(\det\left(\bm{X} \bm{X}^{\top}\right)\right) \; : \; \bm{x}^{1},\ldots,\bm{x}^{k} \in B(\bm{0},1)\right\} \\
= \ \ & \min\left\{- \sum_{i=1}^{d} \log\left(\lambda_{i}\left(\bm{X} \bm{X}^{\top}\right)\right) \; : \; \bm{x}^{1},\ldots,\bm{x}^{k} \in B(\bm{0},1)\right\},
\end{align*}
where $\det(\bm{M})$ denotes the determinant of matrix $\bm{M} \in \mathbb{R}^{d \times d}$.
Note that $f(\lambda_{1},\ldots,\lambda_{d}) = - \sum_{i=1}^{d} \log\left(\lambda_{i}\right)$ is also a symmetric convex function on $(0,\infty)^{d}$.

Another criterion, called E-optimal design, chooses a design that solves
\begin{align*}
& \min\left\{\max\left\{\lambda_{1}\left(\bm{X} \bm{X}^{\top}\right)^{-1},\ldots,\lambda_{d}\left(\bm{X} \bm{X}^{\top}\right)^{-1}\right\} \; : \; \bm{x}^{1},\ldots,\bm{x}^{k} \in B(\bm{0},1)\right\}.
\end{align*}
Note that $f(\lambda_{1},\ldots,\lambda_{d}) = \max\left\{1/\lambda_{1},\ldots,1/\lambda_{d}\right\}$ is also a symmetric convex function on $(0,\infty)^{d}$.
Then the optimization problem for E-optimal design can be written as $$\min\left\{\left(\lambda_{1}\left(\bm{X} \bm{X}^{\top}\right)\right)^{-1} \, : \, \bm{x}^{1},\ldots,\bm{x}^{k} \in B(\bm{0},1)\right\}.$$
For more background about the optimal design of experiments, see \cite{pukelsheim2006optimal}.

In many experimental design settings, including the setting discussed in \Cref{sec:motivation}, an experiment is designed to obtain additional observations after initial observations have already been obtained at design vectors $\bm{u}^{1},\ldots,\bm{u}^{q} \in \mathbb{R}^{d}$.
Thus the overall design will consist of design vectors $\bm{u}^{1},\ldots,\bm{u}^{q},\bm{x}^{1},\ldots,\bm{x}^{k}$, and the overall information matrix will be $\bm{U} \bm{U}^{\top} + \bm{X} \bm{X}^{\top}$, where $\bm{U} = [\bm{u}^{1},\ldots,\bm{u}^{q}] \in \mathbb{R}^{d \times q}$. For notational convenience, let $\bm{A} = \bm{U} \bm{U}^{\top}$.

\subsection{Spectral Design Problem}
\label{subsect:spectral-design-problem}

In this paper, we study the following spectral design problem, which includes all the examples above as special cases.
Let $\mathbb{S}_{d}^+ \subset \mathbb{R}^{d \times d}$ denote the cone of $d \times d$ positive semidefinite matrices. 
Let $f : \mathbb{R}^{d} \to \mathbb{R} \cup \{\infty\}$ be any symmetric convex function, and let the associated spectral function $F : \mathbb{S}_{d}^+ \to \mathbb{R} \cup \{\infty\}$ be given by $F(\bm{M}) \defi f\big(\bm{\lambda}(\bm{M})\big)$.
Then $F$ is unitarily invariant and inherits convexity from $f$ \cite{Lewis1996}.
Our spectral design problem is
\begin{equation}
\label{eq:SpectralDesign}
\min\left\{F\big(\bm{A} + \bm{X} \bm{X}^{\top}\big) \; : \; \bm{x}^{1},\ldots,\bm{x}^{k} \in B(\bm{0},1)\right\}.
\tag{Spectral-Design}
\end{equation}
Although $F : \mathbb{S}_{d}^+ \to \mathbb{R} \cup \{\infty\}$ is convex, the next example shows that $G(\bm{X}) = F\big(\bm{A} + \bm{X} \bm{X}^{\top}\big)$ may be nonconvex.
\begin{example}
\label{ex:Spectral Design nonconvex}
Let $d = 1$, $k = 1$, $A = 0$, and $f(\lambda) = - \lambda$.
Then $G(x) = F(x^{2}) = f(\bm{\lambda}(x^{2})) = f(x^{2}) = - x^{2}$, which is not convex.
Thus \eqref{eq:SpectralDesign} is not a convex optimization problem.
\end{example}

Motivated by the observation that $F : \mathbb{S}_{d}^+ \to \mathbb{R} \cup \{\infty\}$ is convex, consider the following \emph{equivalent} formulation:
\begin{equation}
\label{eq:SpectralDesignM}
\min\left\{F\big(\bm{A} + \bm{M}\big) \; : \; \bm{M} = \sum_{i=1}^{k} \bm{x}^{i} \left(\bm{x}^{i}\right)^{\top}, \; \bm{x}^{1},\ldots,\bm{x}^{k} \in B(\bm{0},1)\right\}.
\tag{Matrix-Spectral-Design}
\end{equation}
The objective of problem~\eqref{eq:SpectralDesignM} is convex.
\begin{example}
\label{ex:Matrix Spectral Design convex with k >= d}
Consider \Cref{ex:Spectral Design nonconvex}.
Note that the feasible set of problem~\eqref{eq:SpectralDesignM} is $\left\{M = x^{2} \, : \, x \in [-1,1]\right\} = [0,1]$, which is convex.
\end{example}
However, as shown in the next example, the feasible set of problem~\eqref{eq:SpectralDesignM} is not convex in general.
\begin{example}
Let $d = 2$ and $k = 1$.
Then $\bm{M}^{1} = \bm{x}^{1} \left(\bm{x}^{1}\right)^{\top}$, where $\bm{x}^{1} = (1,0)^\top$, as well as $\bm{M}^{2} = \bm{x}^{2} \left(\bm{x}^{2}\right)^{\top}$, where $\bm{x}^{2} = (0,1)^\top$, are in the feasible set of problem~\eqref{eq:SpectralDesignM}.
Note that $\bm{M}^{1} / 2 + \bm{M}^{2} / 2 = \bm{I}_{2} / 2$, where $\bm{I}_{2}$ is the identity matrix in $\mathbb{R}^{2 \times 2}$, and thus the feasible set of problem~\eqref{eq:SpectralDesignM} with $d = 2$ and $k = 1$ is not convex.
\end{example}
A standard approach is to optimize over the convex hull of the feasible set $\Big\{\bm{M} \in \mathbb{R}^{d \times d} \, : \, \bm{M} = \sum_{i=1}^{k} \bm{x}^{i} \left(\bm{x}^{i}\right)^{\top}, \, \bm{x}^{1},\ldots,\bm{x}^{k} \in B(\bm{0},1)\Big\}$.
For any set $S \subset \mathbb{R}^{d \times d}$, let $\operatorname{conv}(S)$ denote the convex hull of $S$, that is, the smallest convex set that contains~$S$.

\begin{observation}
\label{obs:conv_feasible_set_M}
The convex hull of the feasible set of problem~\eqref{eq:SpectralDesignM} satisfies
\[
\operatorname{conv}\left\{\bm{M} = \sum_{i=1}^{k} \bm{x}^{i} \left(\bm{x}^{i}\right)^{\top} \; : \; \bm{x}^{1},\ldots,\bm{x}^{k} \in B(\bm{0},1)\right\}
\ \ = \ \
\Big\{\bm{M} \in \mathbb{S}_{d}^+ \; : \; \operatorname{Tr}(\bm{M}) \le k\Big\}.
\]
\end{observation}

Thus, one may consider the following convex relaxation of problem~\eqref{eq:SpectralDesignM}:
\begin{equation}
\label{eq:RelaxedSpectralDesign}
\min\Big\{F\big(\bm{A} + \bm{M}\big) \  : \  \bm{M} \in \mathbb{S}_{d}^+, \  \operatorname{Tr}(\bm{M}) \le k\Big\}.
\tag{Relaxed-Matrix-Spectral-Design}
\end{equation}
However, in spite of problem~\eqref{eq:RelaxedSpectralDesign} being a \emph{tight} convex relaxation of problem~\eqref{eq:SpectralDesignM}, it is not tight enough, in the sense that the optimal objective value of problem~\eqref{eq:RelaxedSpectralDesign} may be strictly less than the optimal objective value of problem~\eqref{eq:SpectralDesignM}, as shown in the next example.

\begin{example}
Let $d = 2$, $k = 1$, $\bm{A} = \bm{I}_{2} / 2$, and $f(\lambda_{1},\ldots,\lambda_{d}) = \max\left\{1/\lambda_{1},\ldots,1/\lambda_{d}\right\}$ as in E-optimal design.
For any $\bm{x} \in B(\bm{0},1)$, the eigenvalues of $\bm{A} + \bm{x} \bm{x}^{\top}$ are $1/2$ and $1/2 + \bm{x}^{\top} \bm{x}$, and the objective value of problem~\eqref{eq:SpectralDesignM} is $F(\bm{A} + \bm{x} \bm{x}^{\top}) = 2$ for all $\bm{x} \in B(\bm{0},1)$.
However, observe that $\bm{M} = \bm{I}_{2} / 2$ is feasible for problem~\eqref{eq:RelaxedSpectralDesign}, with objective value $F(\bm{A} + \bm{M}) = 1$, which cannot be attained in problem~\eqref{eq:SpectralDesignM}.
\end{example}

\subsection{Motivation from Derivative-Free Optimization}
\label{sec:motivation}

Problem \eqref{eq:SpectralDesign} is also motivated by derivative-free optimization (DFO) \cite{Audet2026,conn2009introduction,Larson2019,Rios2013}. 
Consider the unconstrained problem
\begin{align}
\label{eq:unconstrained-problem}
\min_{\bm{y} \in \mathbb{R}^{d}} \, g(\bm{y}).
\end{align}
In DFO, function evaluations are typically subject to errors, for example, due to round-off effects or systematic evaluation inaccuracies. For $\bm{y} \in \mathbb{R}^{d}$, let $\tilde{g}(\bm{y})$ denote the computationally available approximation of $g(\bm{y})$.
Suppose that there exists a finite constant $\varepsilon_{\mathrm{abs}} > 0$ such that
\begin{align}
\label{eq:inexact-function-evaluations}
\left|\left[g(\bm{v}+\bm{w}) - g(\bm{v})\right] - \left[\tilde{g}(\bm{v}+\bm{w}) - \tilde{g}(\bm{v})\right]\right|
\le \varepsilon_{\mathrm{abs}}
\quad \text{for all } \bm{v}, \bm{w} \in \mathbb{R}^{d},
\end{align}
that is, differences are approximated within a tolerance.

In many scientific and engineering applications, a software tool, such as \texttt{EnergyPlus}
\cite{Woods2023}, has been developed that takes $\bm{y}$ (and possibly additional parameters) as input and returns $\tilde{g}(\bm{y})$.
Even if $g$ is differentiable, such tools may not output the gradient $\nabla g(\bm{y})$.
This lack of derivative information may arise for several reasons: implementing gradient computations may be time-consuming; the code structure may prevent the use of automatic differentiation or of the complex step method; or practitioners may wish to avoid possible discrepancies between code that computes $g$ and code that computes $\nabla g$.
Moreover, each evaluation of $\tilde{g}(\bm{y})$ is usually computationally expensive.
A similar situation occurs when $\tilde{g}(\bm{y})$ is obtained from physical experiments or complex simulations.
Examples of such settings include
building energy simulation \cite{Dogan2025},
urban transportation problems
\cite{Osorio2013}, 
groundwater remediation
\cite{Yoon1999}, and 
helicopter rotor design
\cite{Booker1998}.

To address such problems, numerous derivative-free (or zeroth-order) optimization algorithms have been developed.
Many of these methods construct local surrogate models of $g$ using evaluations at selected design vectors.
As function evaluations are expensive, careful selection of these points is crucial.
The new design vectors are of the form $\bm{y} + \delta \bm{x}^{1}, \ldots, \bm{y} + \delta \bm{x}^{k}$, where $\bm{y}$ is the incumbent solution, $\delta > 0$ is the design-region radius, and $\bm{x}^{1}, \ldots, \bm{x}^{k} \in B(\bm{0},1)$ are chosen design directions.
Thus, the current design region for new evaluations is the Euclidean ball of radius $\delta$ centered at $\bm{y}$.
Model-based DFO methods construct, at each iteration, a surrogate model $\hat{g} : \mathbb{R}^{d} \to \mathbb{R}$ using observations at newly selected design vectors, as well as observations at previously chosen design vectors.
It is advantageous to reuse some function evaluations from previous iterations, because evaluations of $\tilde{g}(\bm{y} + \delta \bm{x}^{i})$ are costly.
For example, if function values are reused at points whose distance from $\bm{y}$ is at most $r \delta$, where $r > 0$ is the reuse radius, then the reused points can be written as $\bm{y} + \delta \bm{u}^{j}$, with $\bm{u}^{j} \in B(\bm{0},r) \setminus \{\bm{0}\}$, $j=1,\ldots,q$, where $q$ is the number of reused points.
\Cref{fig:dfo-design} illustrates design vectors in two successive iterations, where the function evaluation at the point $\bm{y}' + \delta' \bm{x}^{\prime 1}$ from the first iteration is reused as a function evaluation at $\bm{y} + \delta \bm{u}^{1}$ in the next iteration.

Under standard polynomial regression models used in DFO (e.g., first- or second-order Taylor approximations), the quality of the surrogate $\hat{g}$ depends critically on the geometry of the reused design directions $\bm{u}^{1},\ldots,\bm{u}^{q}$ and the new design directions $\bm{x}^{1},\ldots,\bm{x}^{k}$.
If $\bm{U} \defi [\bm{u}^{1},\ldots,\bm{u}^{q}]$, then the associated information matrix takes the form $\bm{A} + \sum_{i=1}^{k} \bm{x}^{i} (\bm{x}^{i})^{\top}$, where $\bm{A} \defi \bm{U} \bm{U}^{\top}$ represents the contribution of the previous design vectors to the new design matrix.
The optimal selection of new design vectors can therefore be formulated as a spectral design problem of the form \eqref{eq:SpectralDesign}, where the choice of spectral function $F$ corresponds to the desired optimality criterion (e.g., A-, D-, or E-optimality).

\begin{figure}
\centering
\begin{tikzpicture}[scale=1,>=Latex]

\coordinate (yprev) at (-0.05,-1.30);
\coordinate (ycurr) at (3.4,-0.5);

\def\deltaR{1.25}
\def\reuseR{2.34}

\coordinate (p1) at ($(yprev)+(\deltaR,0)$);
\coordinate (p2) at ($(yprev)+(0,\deltaR)$);
\coordinate (n1) at ($(ycurr)!\deltaR cm!-90:(p1)$);

\fill[oldgray, opacity=0.03] (yprev) circle (\deltaR);
\fill[newblue, opacity=0.04] (ycurr) circle (\deltaR);
\fill[neworange, opacity=0.03] (ycurr) circle (\reuseR);

\draw[dashed, line width=0.9pt, oldgray]
  (yprev) circle (\deltaR);
\draw[dashed, line width=0.9pt, newblue]
  (ycurr) circle (\deltaR);
\draw[dashed, line width=0.9pt, neworange]
  (ycurr) circle (\reuseR);

\filldraw[black] (yprev) circle (1.6pt);
\filldraw[black] (ycurr) circle (1.6pt);

\node[below left=2pt] at (yprev) {$\bm{y}'$};
\node[below right=2pt] at (ycurr) {$\bm{y}$};

\draw[->, line width=0.9pt, dashed] (yprev) -- (ycurr);

\draw[->, line width=0.9pt, oldgray] (yprev) -- (p1);
\draw[->, line width=0.9pt, oldgray] (yprev) -- (p2);

\node[oldgray, below=2pt] at ($(yprev)!0.55!(p1)$)
  {$\delta' \bm{x}^{\prime 1}$};
\node[oldgray, left=2pt] at ($(yprev)!0.55!(p2)$)
  {$\delta' \bm{x}^{\prime 2}$};

\draw[->, line width=0.9pt, neworange] (ycurr) -- (p1);
\node[neworange, below=2pt] at ($(ycurr)!0.65!(p1)$)
  {$\delta \bm{u}^{1}$};

\draw[->, line width=0.9pt, newblue] (ycurr) -- (n1);
\node[newblue, right=2pt] at ($(ycurr)!0.55!(n1)$)
  {$\delta \bm{x}^2$};

\filldraw[neworange] (p1) circle (2.2pt);
\filldraw[oldgray] (p2) circle (2.2pt);
\filldraw[newblue] (n1) circle (2.2pt);

\end{tikzpicture}
\caption{Reuse of previously sampled points after the incumbent moves from $\bm{y}'$ to $\bm{y}$.
The gray circle is the previous radius-$\delta'$ design region, the blue circle is the current radius-$\delta$ design region, and the orange circle is the reuse region of radius $r \delta$.
The point $\bm{y}' + \delta' \bm{x}^{\prime 1}$ is reused at $\bm{y}$ as the displacement $\delta \bm{u}^{1}$, whereas $\bm{y}' + \delta' \bm{x}^{\prime 2}$ is not reused.
The vector $\delta \bm{x}^2$ denotes a newly selected design direction.}
\label{fig:dfo-design}
\end{figure}
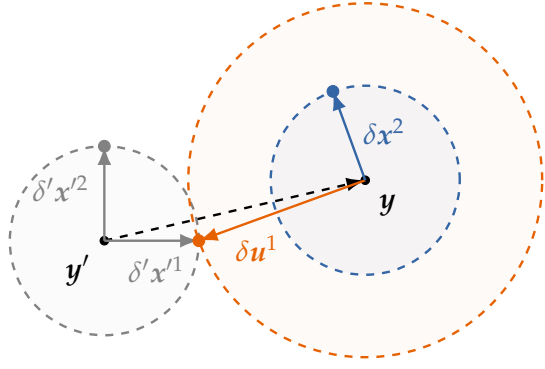

\paragraph{First-Order Models.}
A widely used local model of $g$ is the linear model
\[
\hat{g}(\bm{y} + \bm{x}) \ \ = \ \ \tilde{g}(\bm{y}) + \bm{x}^{\top} \nabla \hat{g}(\bm{y}).
\]
At the time that the model $\hat{g}$ is constructed, the DFO algorithm has already evaluated $g$ at the incumbent solution $\bm{y}$, and thus $\tilde{g}(\bm{y})$ is given.
Suppose additionally that $q$ stored function values $\tilde{g}(\bm{y} + \delta \bm{u}^{j})$, $j=1,\ldots,q$, are reused.
In contrast, $\nabla \hat{g}(\bm{y})$ is unknown and must be estimated.
The design directions $\bm{x}^{1},\ldots, \bm{x}^{k}$ are selected, and the function values $\tilde{g}(\bm{y} + \delta \bm{x}^{1}), \ldots, \tilde{g}(\bm{y} + \delta \bm{x}^{k})$ are computed.
Using both the reused and newly obtained values, the estimated gradient $\nabla \hat{g}(\bm{y})$ is obtained as the solution of the least-squares problem
\begin{align*}
& \min_{\nabla \hat{g}(\bm{y}) \in \mathbb{R}^{d}}
\Bigg\{\sum_{j=1}^{q} \left(\delta \, \big(\bm{u}^{j}\big)^{ \top} \nabla \hat{g}(\bm{y}) - \left[\tilde{g}(\bm{y} + \delta \bm{u}^{j}) - \tilde{g}(\bm{y})\right]\right)^{2} \\*
& \qquad \qquad \qquad + \sum_{i=1}^{k} \left(\delta \, \big(\bm{x}^{i}\big)^{ \top} \nabla \hat{g}(\bm{y}) - \left[\tilde{g}(\bm{y} + \delta \bm{x}^{i}) - \tilde{g}(\bm{y})\right]\right)^{2}\Bigg\}.
\end{align*}
For $q=0$, it reduces to the least-squares problem
\begin{align*}
\min_{\nabla \hat{g}(\bm{y}) \in \mathbb{R}^{d}}
\sum_{i=1}^{k} \left(\delta \, \big(\bm{x}^{i}\big)^{ \top} \nabla \hat{g}(\bm{y}) - \left[\tilde{g}(\bm{y} + \delta \bm{x}^{i}) - \tilde{g}(\bm{y})\right]\right)^{2}.
\end{align*}

For DFO, it is important to control the gradient error $\|\nabla \hat{g}(\bm{y}) - \nabla g(\bm{y})\|_{2}$.
Assume that there exists a constant $L_{\nabla g} > 0$ such that
\[
\left\|\nabla g(\bm{v} + \bm{w}) - \nabla g(\bm{v})\right\|_{2} \ \ \le \ \ L_{\nabla g} \|\bm{w}\|_{2}
\quad \text{for all} \ \bm{v}, \bm{w} \in \mathbb{R}^{d}.
\]
Also, assume that $[\bm{U},\bm{X}]$ has full row rank.
Using the inexact evaluation model \eqref{eq:inexact-function-evaluations}, the Lipschitz continuity of $\nabla g$, and standard perturbation bounds for least-squares problems (see, e.g., \cite[Prop.~3.2]{Demmel1997}), it follows that
\begin{align*}
\left\|\nabla \hat{g}(\bm{y}) - \nabla g(\bm{y})\right\|_{2}
\ \ \le \ \ & \left(\lambda_{d}\left(\left(\bm{A} + \bm{X} \bm{X}^{\top}\right)^{-1}\right)\right)^{1/2} \\
& \quad \cdot \left(q \left(\frac{L_{\nabla g} r^{2} \delta}{2} + \frac{\varepsilon_{\mathrm{abs}}}{\delta}\right)^{2}
+ k \left(\frac{L_{\nabla g} \delta}{2} + \frac{\varepsilon_{\mathrm{abs}}}{\delta}\right)^{2}\right)^{1/2}.
\end{align*}
Minimizing the right-hand side over $\delta > 0$ yields
\[
\delta \ \ = \ \ \sqrt{\frac{2 \varepsilon_{\mathrm{abs}}}{L_{\nabla g}}} \left(\frac{q + k}{q r^{4} + k}\right)^{1/4}.
\]
If $q = 0$, then $\delta = (2 \varepsilon_{\mathrm{abs}} / L_{\nabla g})^{1/2}$, which is similar to eq.~(2.3) in \cite{Berahas2022}, although that result is derived under an error model slightly different from \eqref{eq:inexact-function-evaluations}.
Thus, for fixed $q$, $r$, and $\delta$, the resulting design problem becomes
\[
\min\left\{\lambda_{d}\left(\left(\bm{A} + \bm{X} \bm{X}^{\top}\right)^{-1}\right) \; : \; \bm{x}^{i} \in B(\bm{0},1) \; \forall \ i = 1,\ldots,k\right\}.
\]
This problem is a special case of E-optimal design with overall information matrix $\bm{A} + \bm{X} \bm{X}^{\top}$.

\subsection{Our Contributions}

Despite the challenges posed by natural convex optimization approaches to solving the problem~\eqref{eq:SpectralDesign} or its equivalent reformulation~\eqref{eq:SpectralDesignM}, we present an alternative approach and show that it can be efficiently reformulated as a convex optimization problem.
The main difference from the standard approach is that we consider an optimization problem in which the \emph{eigenvalues} of $\bm{A} + \bm{M}$ are the decision variables.
In addition to the trace constraints, we add constraints based on Weyl's eigenvalue inequalities~\cite{horn2012matrix} that use the fact that the rank of $\bm{M}$ is at most $\min\{d,k\}$.
The resulting convex optimization problem is a relaxation of \eqref{eq:SpectralDesignM}, in the sense that for every feasible solution~$\bm{M}$ of~\eqref{eq:SpectralDesignM}, the eigenvalues of $\bm{A} + \bm{M}$ give a feasible solution of the convex optimization problem, with objective value equal to $F(\bm{A} + \bm{M})$.
While generic convex optimization algorithms can be used to solve this problem, we provide a simple and efficient algorithm to compute an optimal solution.
Then we provide an algorithm that converts the resulting eigenvalues to an optimal design matrix $\bm{X}=[\bm{x}^{1},\ldots,\bm{x}^{k}]$ for problem~\eqref{eq:SpectralDesign} in $O({k}d + d^3)$ steps if $f$ is non-increasing and an $\epsilon$-optimal design if $f$ is Lipschitz continuous.
This algorithm relies on an algorithmic version of the Schur--Horn theorem~\cite{horn2012matrix}.
This polynomial-time method contrasts with optimal design problems over finite candidate sets, where computing an optimal design for typical criteria such as A-, D-, or E-optimality is NP-hard \cite{vcerny2012two,civril2009selecting}.

\begin{theorem}
\label{thm:general_opt_design}
Let  $f : \mathbb{R}^{d} \to \mathbb{R} \cup \{\infty\}$ be symmetric and convex.
Then the following statements hold.
\begin{enumerate}[nosep, label=\arabic*.]
\item If $f$ is non-increasing, then \Cref{alg:computing-optimal-design} returns an optimal solution of
problem~\eqref{eq:SpectralDesign} in runtime $O(d^3+kd)$.
Moreover, the
returned solution is independent of the particular choice of $f$.
\item If $f$ is Lipschitz continuous with Lipschitz constant $L > 0$ on
the compact eigenvalue region
\[
\left\{
\bm{\lambda}\left(\bm A+\bm X\bm X^{\top}\right)
:
\bm x^1,\ldots,\bm x^k\in B(\bm 0,1)
\right\},
\]
then given a value oracle for $f$,
for any fixed $\varepsilon>0$, 
\Cref{alg:computing-optimal-design} returns an $\varepsilon$-optimal
solution of problem~\eqref{eq:SpectralDesign} with $O\!\left(\log(kL/\varepsilon)\right)$ oracle calls to the function $f$ and runtime
\[
O\!\left(d^3+kd+d\log(kL/\varepsilon)\right).
\]
\end{enumerate}
\end{theorem}

 We emphasize that the first result in Theorem~\ref{thm:general_opt_design} guarantees that given the prior information matrix $\bm{A}$ and target number of design vectors $k$, there is a \emph{single} set of design vectors that are optimal for every non-increasing $f$ that the algorithm efficiently computes. This captures all the natural criteria, such as  A-, D-, or E-optimality, that are defined by a non-increasing $f$. 

In light of our results, we revisit convex reformulations for problem \eqref{eq:SpectralDesignM} and obtain the following corollary, which identifies the feasible set of \eqref{eq:SpectralDesignM} with a rank-constrained set. 
\begin{corollary}
\label{cor:spectral-design-rank}
The feasible set of problem~\eqref{eq:SpectralDesignM} is given by
\[
\left\{\bm{M} = \sum_{i=1}^{k} \bm{x}^{i} \left(\bm{x}^{i}\right)^{\top} \  : \  \bm{x}^{1},\ldots,\bm{x}^{k} \in B(\bm{0},1)\right\}
\ \ = \ \ \Big\{\bm{M} \in \mathbb{S}_{d}^+ \  : \  \operatorname{Tr}(\bm{M}) \le k, \  \operatorname{rank}(\bm{M}) \le k\Big\}.
\]

\end{corollary}
Observe that \Cref{cor:spectral-design-rank} implies that, when $k \geq d$, \eqref{eq:RelaxedSpectralDesign} is an exact reformulation of \eqref{eq:SpectralDesign} because the rank constraint is superfluous.
Also, when $k < d$, this gives a nonconvex rank-constrained reformulation of \eqref{eq:SpectralDesign} and \eqref{eq:SpectralDesignM}.

In addition to the theoretical results, we provide numerical
illustrations in a DFO setting. On standard test problems, we incorporate
the proposed spectral designs into a first-order regression-based DFO
method and compare the resulting method with a finite-difference variant
and a coordinate-based regression variant.

\subsection{Related Literature}

Optimal experimental design is a classical statistical problem that aims to select factor values to minimize the variance of an estimator or to maximize the Fisher information under a parametric model \cite{pukelsheim2006optimal}. A classical continuous (or approximate) design framework \cite{kiefer1985collected} represents a design as a probability measure over a design space, leading to convex optimization over information matrices for criteria such as A-, D-, and E-optimality, which are functions of the eigenvalues of the information matrix. Foundational monographs such as \cite{pukelsheim2006optimal} develop this theory rigorously, emphasizing convexity, Loewner ordering, and equivalence theorems.
In \cite{kiefer1960equivalence,whittle1973some}, continuous optimal design is formulated through directional derivative characterizations and convex optimization approaches that exploit the spectral structure of the information matrix. Most existing methods are tailored to a specific design criterion (e.g., A- or D-optimality). In contrast, the spectral design results in this paper hold for a broad class of convex spectral criteria.

When the candidate design vectors are finite, the optimal design problem becomes combinatorial. Exact experimental design aims to select a given number of design vectors from a finite candidate set, resulting in a discrete optimization problem with a convex objective function over integer allocation constraints \cite{Sagnol2015}.
Recent advances in mathematical programming have enabled exact computation for classical criteria. For example, mixed-integer second-order cone programming formulations for D-optimal design \cite{sagnol2011computing,Sagnol2015} provide provably optimal solutions and extend to related criteria such as A- and G-optimality. Branch-and-bound and branch-and-cut methods have also been developed to solve combinatorial design problems to global optimality \cite{hendrych2023solving,pillai2025computing,ponte2025relationship,wang2025d}.
Approximation algorithms for cardinality-constrained A- and D-optimal design using sampling, rounding, and volume sampling techniques have been proposed in \cite{madan2019combinatorial,nikolov2022proportional,singh2020approximation,wang2025algorithms}, offering scalable methods with performance guarantees for large candidate sets.
This paper considers a convex design space rather than a finite candidate set.

Across both continuous and combinatorial design spaces, many design criteria are spectral functions of information matrices.
The convex analysis of spectral functions \cite{lewis1995convex} provides the mathematical foundation for understanding such objectives, as spectral convexity and majorization theory underpin equivalence theorems, subgradient characterizations, and relaxation techniques \cite{kim2022convexification,li2025partial}.
\cite{kiefer1975construction,pukelsheim2006optimal} show how eigenvalue structure captures design efficiency across a family of convex objectives. This work builds directly on this spectral viewpoint and shows that a single design can optimize a broad class of symmetric convex spectral criteria.

DFO addresses optimization problems with an oracle that provides (exact or approximate) objective values at specified points, but that does not provide derivatives.
Model-based methods construct local linear or quadratic surrogate models from function values provided by the oracle. 
The monograph \cite{conn2009introduction} develops the geometry-of-sampling perspective, introduces the notion of poisedness, and shows that the conditioning and spectral properties of sampling matrices govern stability and error bounds for interpolation and regression models.
This directly connects DFO sampling to experimental design \cite{wild2013global}: selecting directions that control the eigenvalues of matrices such as $\bm{X} \bm{X}^{\top}$ affects the quality of gradient and Hessian estimates, particularly under noisy or inexact evaluations.
Thus, spectral criteria, such as minimizing $\lambda_{d}((\bm{X} \bm{X}^{\top})^{-1})$ or trace-inverse surrogates, arise naturally as objectives for constructing sampling sets in DFO \cite{bandeira2012computation}.
Unlike the existing literature, which typically focuses on specific criteria or heuristic constructions, this paper optimally selects sampling directions using a fast, scalable algorithm.

\section{Lower Bounds for Spectral Design}

First, we derive a convex optimization problem, with eigenvalues as decision variables, such that the optimal objective value gives a lower bound for problem~\eqref{eq:SpectralDesign}.
Unlike the relaxation~\eqref{eq:RelaxedSpectralDesign}, the feasible set of this optimization problem depends not only on $d$ and $k$, but also on $\bm{A}$.
We give a simple, efficient algorithm to compute an optimal solution for this optimization problem.
In the next section, we show how to compute a feasible design that attains the lower bound, and that is therefore optimal for problem~\eqref{eq:SpectralDesign}.

\subsection{The Lower Bound}

Since orthonormal transformations preserve eigenvalues and Euclidean norms, problem \eqref{eq:SpectralDesign} is invariant under orthonormal changes of coordinates.
Therefore, without loss of generality, we assume that $\bm{A}$ is diagonal.
Let $\bm{A} = \operatorname{diag}(t_{1},\ldots,t_{d})$ with $0 \leq t_{1} \le t_{2} \le \cdots \le t_{d}$.

We use Weyl's eigenvalue inequalities to construct the lower bound.
\begin{theorem}[Weyl's Eigenvalue Inequalities, {see \cite[Theorem 4.3.1]{horn2012matrix}}]
\label{thm_weyl}
Let $\bm{B},\bm{C} \in \mathbb{S}^{d}$ be two matrices with eigenvalues $\lambda_{1}(\cdot) \le \cdots \le \lambda_{d}(\cdot)$ in non-decreasing order.
Then for all $i,j$ satisfying $1 \le i,j$ and $i+j-1 \le d$, it holds that
\[
\lambda_{i+j-1}(\bm{B} + \bm{C}) \ \ \geq \ \ \lambda_{i}(\bm{B})+\lambda_{j}(\bm{C})
\]
and for all $i,j$ satisfying $1 \le i,j$ and $i+j-d \ge 1$,
\[
\lambda_{i+j-d}(\bm{B} + \bm{C}) \ \ \le \ \ \lambda_{i}(\bm{B}) + \lambda_{j}(\bm{C}).
\]
\end{theorem}

For any $\bm{X}$ feasible for~\eqref{eq:SpectralDesign}, let us first consider the case with $k < d$.
Since $\bm{X} \in \mathbb{R}^{d \times k}$, the matrix $\bm{X} \bm{X}^{\top}$ has rank at most $k$.
Moreover,
\[
\sum_{i=1}^{d} \lambda_{i}(\bm{X} \bm{X}^{\top}) \ \ = \ \ \operatorname{Tr}(\bm{X} \bm{X}^{\top}) \ \ \le \ \ k.
\]
It follows from \Cref{thm_weyl} with $\bm{B} = \bm{A}$ and $\bm{C} = \bm{X} \bm{X}^{\top}$ that
\begin{equation}
\lambda_{j}(\bm{A} + \bm{X} \bm{X}^{\top}) \ \ \ge \ \ \lambda_{j}(\bm{A}) \ \ = \ \ t_{j}
\qquad \forall \ j=1,\ldots,d,
\label{eqn:lower_bound_weyl}
\end{equation}
and
\begin{equation*}
\lambda_{j}(\bm{A} + \bm{X} \bm{X}^{\top}) \ \ \le \ \ \lambda_{j + k}(\bm{A}) + \lambda_{d - k}(\bm{X} \bm{X}^{\top})
\qquad \forall \ j=1,\ldots,d-k.
\end{equation*}
Since $\bm{X} \bm{X}^{\top}$ has rank at most $k$, its smallest $d - k$ eigenvalues are zero, and thus $\lambda_{d - k}(\bm{X} \bm{X}^{\top}) = 0$.
Hence, 
\begin{equation}
\lambda_{j}(\bm{A} + \bm{X} \bm{X}^{\top}) \ \ \le \ \ \lambda_{j + k}(\bm{A}) \ \ = \ \ t_{j + k}
\qquad \forall \ j=1,\ldots,d-k.
\label{eqn:upper_bound_weyl}
\end{equation}
Let
\[
\beta_{j} \ \ \defi \ \ \lambda_{j}(\bm{A} + \bm{X} \bm{X}^{\top}) - t_{j}
\qquad \forall \ j=1,\ldots,d.
\]
It follows from \eqref{eqn:lower_bound_weyl} and \eqref{eqn:upper_bound_weyl} that $\bm{\beta}$ satisfies
\[
\beta_{j} \ \ \ge \ \ 0 \ \ \forall \ j=1,\ldots,d, \qquad
t_{j} + \beta_{j} \ \ \le \ \ t_{j+k} \ \ \forall \ j=1,\ldots,d-k, \qquad
\sum_{j=1}^{d} \beta_{j} \ \ = \ \ \operatorname{Tr}(\bm{X} \bm{X}^{\top}) \ \ \le \ \ k.
\]

Next consider the case with $k \ge d$.
In this case, the Weyl lower bounds $\beta_{j} \ge 0$, $j=1,\ldots,d$, and the trace constraint $\sum_{j=1}^{d} \beta_{j} \le k$ remain valid, whereas the Weyl upper-bound constraints \eqref{eqn:upper_bound_weyl} are vacuous.

The preceding arguments yield the following lower bound on the optimal value of problem~\eqref{eq:SpectralDesign}.
\begin{proposition}
\label{prop:lower_bound_general}
Let $\bm{A} = \operatorname{diag}(t_{1},\ldots,t_{d})$ with $0 \leq t_{1} \le \cdots \le t_{d}$, and let $\hat{d} \defi \min\{d,k\}$.
Then the optimal value of problem~\eqref{eq:SpectralDesign} is bounded below by the optimal value of
\begin{equation}
\label{eqn:relaxation-general}
\min_{\bm{\beta} \in \mathbb{R}_+^{d}} \left\{f(t_{1} + \beta_{1},\ldots,t_{d} + \beta_{d}) \; : \; t_{j} + \beta_{j} \le t_{j+\hat{d}} \  \forall \, j=1,\ldots,d-\hat{d}, \ \sum_{j=1}^{d} \beta_{j} \le k\right\}.
\tag{Relaxed-Eigen}
\end{equation}
\end{proposition}

\begin{proof}
For any feasible $\bm{X}$ of problem~\eqref{eq:SpectralDesign}, it follows from Weyl's inequalities that the vector $\bm{\beta} = \bm{\lambda}(\bm{A} + \bm{X} \bm{X}^{\top}) - \bm{\lambda}(\bm{A})$ is feasible for problem \eqref{eqn:relaxation-general}.
Also, $F(\bm{A} + \bm{X} \bm{X}^{\top}) = f(t_{1} + \beta_{1},\ldots,t_{d} + \beta_{d})$, and thus the objective value of $\bm{X}$ in~\eqref{eq:SpectralDesign} is equal to the objective value of $\bm{\beta}$ in~\eqref{eqn:relaxation-general}.
Thus, the optimal value of problem~\eqref{eqn:relaxation-general} lower bounds the optimal value of problem~\eqref{eq:SpectralDesign}.
\end{proof}

\Cref{prop:lower_bound_general} relaxes the matrix-valued optimization problem~\eqref{eq:SpectralDesign} to a finite-dimensional convex program in the eigenvalue increments $\bm{\beta}$.
The relaxation preserves the spectral structure of the original problem: it captures the fact that adding $\bm{X} \bm{X}^{\top}$ can increase eigenvalues of $\bm{A}$, and that the total increase is limited by $\operatorname{Tr}(\bm{X} \bm{X}^{\top}) \le k$ and by the rank constraint $\operatorname{rank}(\bm{X} \bm{X}^{\top}) \le \hat{d}$. The lower-bound problem depends only on the eigenvalues $\bm{t}$ of the prior matrix $\bm{A}$ and on the number $k$ of new design vectors. It no longer depends on a particular factorization of the rank-one update.
As shown in the subsequent sections, this lower bound is tight: we construct a design matrix $\bm{X}$ that attains it. This tightness is the key step toward proving \Cref{thm:general_opt_design} and reveals that the optimal design can be characterized entirely by eigenvalue allocation rather than by the geometry of design vectors.

Next we introduce some fundamental definitions about majorization and Schur-convex functions that will be useful for showing the optimality of a solution constructed for problem~\eqref{eqn:relaxation-general}.

\paragraph{Majorization and Schur-Convexity.} 
We follow the standard conventions in majorization theory; see, for example, \cite{marshall1979inequalities}.
For a vector $\bm{x} \in \mathbb{R}^{d}$, let $\bm{x}^{\downarrow}$ denote the vector obtained by sorting the coordinates of $\bm{x}$ in non-increasing order.

\begin{definition}[Majorization]
\label{def:majorization}
For $\bm{x},\bm{y}\in\mathbb{R}^{d}$, we say that $\bm{x}$ is \emph{majorized by}
$\bm{y}$, and write $\bm{x} \preceq \bm{y}$, if
\begin{equation}
\label{eqn:majorization-definition}
\sum_{i=1}^{m} x_{i}^{\downarrow} \ \ \le \ \ \sum_{i=1}^m y_{i}^{\downarrow}
\qquad \forall \ m=1,\ldots,d-1,
\end{equation}
and
\begin{equation}
\label{eqn:majorization-equal-sum}
\sum_{i=1}^{d} x_{i} \ \ = \ \ \sum_{i=1}^{d} y_{i}.
\end{equation}
In words, $\bm{x} \preceq \bm{y}$ means that $\bm{x}$ is less spread out than $\bm{y}$ while having the same total mass.
\end{definition}

Functions that are monotone with respect to the majorization order are called Schur-convex, as introduced below.
\begin{definition}[Schur-convexity]
\label{def:schur-convexity}
A function $f : \mathbb{R}^{d} \to \mathbb{R} \cup \{\infty\}$ is called \emph{Schur-convex} if
\begin{equation}
\label{eqn:schur-convex-definition}
\bm{x} \preceq \bm{y}
\quad \Longrightarrow \quad
f(\bm{x}) \le f(\bm{y}).
\end{equation}
\end{definition}
The following is a standard result about majorization and Schur-convexity (see~\cite[Proposition~C.2]{marshall1979inequalities}).\footnote{Although \cite[Proposition~C.2]{marshall1979inequalities}
is stated for real-valued functions, its conclusion also holds for
extended-real-valued convex functions. Indeed, by the convex-hull
characterization of majorization
\cite[pp.~8--9]{marshall1979inequalities}, if $\bm{x}\preceq\bm{y}$,
then $\bm{x}$ lies in the convex hull of the permutations of $\bm{y}$.
If $f(\bm{y}) < \infty$, then symmetry and convexity imply that $f(\bm{x})\le f(\bm{y})$.
If $f(\bm{y}) = \infty$, then $f(\bm{x})\le f(\bm{y})$ trivially holds.}
\begin{lemma}
\label{lem:majorization-schur-convexity}
If $f : \mathbb{R}^{d} \to \mathbb{R} \cup \{\infty\}$ is a symmetric, convex function, then $f$ is Schur-convex.
Thus, if a feasible vector is majorized by every other feasible vector, then it minimizes every symmetric convex objective over that feasible set.
\end{lemma}

\subsection{An Optimal Solution to the Lower Bound Problem: non-increasing Case}
\label{sec_opt_lb:monotone}

In this section, we give a simple algorithm to compute an optimal solution for problem~\eqref{eqn:relaxation-general} when the function $f$ is non-increasing.
In the next subsection, we build on this result to give an algorithm for the general case.

It is convenient to write the Weyl upper bounds as capacities.
Recalling $\hat{d} = \min\{d,k\}$, let
\begin{equation}
\label{eqn:caps-water-filling}
u_{j} \ \ \defi \ \
\begin{cases}
t_{j + \hat{d}}, & j = 1,\ldots,d-\hat{d},\\[1mm]
\infty, & j = d-\hat{d}+1,\ldots,d.
\end{cases}
\end{equation}
Then \eqref{eqn:relaxation-general} becomes
\begin{equation}
\label{eqn:box-relaxation}
\min_{\bm{\beta} \in \mathbb{R}_+^{d}}
\left\{f(\bm{t} + \bm{\beta}) \; : \; 0 \le \beta_{j} \le u_{j} - t_{j} \ \forall \ j = 1, \ldots, d, \ \sum_{j=1}^{d} \beta_{j} \le k\right\}.
\end{equation}
The variables $t_{j} + \beta_{j}$ may be interpreted as water levels in buckets: $t_{j}$ is the initial level, $u_{j}$ is the capacity, and $k$ is the total amount of water available.

For any $k \ge 1$ and $c \ge t_{1}$, let
\begin{equation}
\label{eqn:Phi-water-filling}
\Phi_{k}(c) \ \ \defi \ \ \sum_{j=1}^{d} \min\big\{(c-t_{j})_+,\,u_{j}-t_{j}\big\},
\qquad \mbox{where } (x)_+ \defi \max\{x,0\}.
\end{equation}
Thus $\Phi_{k}(c)$ is the amount of water required to raise every bucket to level $c$, truncating bucket~$j$ at its capacity $u_{j}$ if $u_{j} \le c$
($\Phi_{k}$ depends on~$k$ because $\bm{u}$ depends on~$k$).
Since the final $\hat{d}$ buckets have infinite capacity, it follows that $\Phi_{k}$ is increasing and $\Phi_{k}(c) \to \infty$ as $c \to \infty$.
For any $k \ge 1$, let
\begin{equation}
\label{eqn:ck-definition}
c(k) \ \ \defi \ \ \sup\{c \; : \; c \ge t_{1}, \; \Phi_{k}(c) \le k\}.
\end{equation}
(The supremum above is attainable for all $k \ge 1$ because $\Phi_{k}(t_{1}) = 0$ and $\Phi_{k}$ is continuous and increasing and $\Phi_{k}(c) \to \infty$ as $c \to \infty$.)
Thus, $c(k)$ is the water level reached after pouring $k$ units of water into the buckets.
Let
\begin{equation}
\label{eqn:beta-water-filling}
\beta_{j}^* \ \ \defi \ \ \min\big\{(c(k)-t_{j})_+, \, u_{j}-t_{j}\big\},
\qquad j=1,\ldots,d.
\end{equation}
The vector $\bm{\beta}^*$ is obtained by raising the levels in the buckets with the least initial amounts of water in lockstep until the budget~$k$ is exhausted.
Note that, given the eigenvalues $t_{1},\ldots,t_{d}$ in the sorted order, $c(k)$ and $\bm{\beta}^*$ can be computed in linear time.

\begin{proposition}
\label{prop:water-filling-monotone}
Suppose that $f$ is a symmetric, convex, and non-increasing function.
Then the vector $\bm{\beta}^*$ defined in \eqref{eqn:beta-water-filling} is an optimal solution of \eqref{eqn:box-relaxation}, and hence of \eqref{eqn:relaxation-general}.
\end{proposition}
\begin{proof}
It follows from $f$ being non-increasing and $u_{d} = \infty$, that we can restrict attention to feasible $\bm{\beta}$ that satisfies $\sum_{j=1}^{d} \beta_{j} = k$.
Then the set of values of $\bm{t} + \bm{\beta}$ for such $\bm{\beta}$ equals 
\[
\mathcal{Y}_{k} \ \ \defi \ \ \left\{\bm{y} \in \mathbb{R}^{d} \; : \; t_{j} \le y_{j} \le u_{j} \ \forall \ j=1,\ldots,d, \ \sum_{j=1}^{d} y_{j} = \sum_{j=1}^{d} t_{j} + k\right\}.
\]
Consider $\bm{y}^*$ given by ${y}_{j}^* = t_{j} + \beta^*_{j}$ for each $j$.
Note that $\bm{y}^* \in \mathcal{Y}_{k}$.
We will show that
\begin{equation}
\label{cl:ordering}
\bm{y}^* \ \preceq \ \bm{y} \quad \forall \ \bm{y} \in \mathcal{Y}_{k}.
\end{equation}
It follows from \Cref{lem:majorization-schur-convexity} that $f$ is Schur-convex.
Thus it follows from~\eqref{cl:ordering} that $f(\bm{y}^*) \le f(\bm{y})$ for all $\bm{y} \in \mathcal{Y}_{k}$.
Consider any feasible $\bm{\beta}$ that satisfies $\sum_{j=1}^{d} \beta_{j} = k$, let $\bm{y} = \bm{t} + \bm{\beta}$, and observe that $\bm{y} \in \mathcal{Y}_{k}$.
Then $f(\bm{t} + \bm{\beta}^*) = f(\bm{y}^*) \le f(\bm{y}) = f(\bm{t} + \bm{\beta})$, showing that $\bm{\beta}^*$ is optimal.

Next, we show \eqref{cl:ordering}.
Consider any $\bm{y} \in\mathcal{Y}_{k} \setminus \bm{y}^*$. 
Since $\sum_{j=1}^{d} y_{j} = \sum_{j=1}^{d} y_{j}^*$, there exist indices $p,q$ such that $y_{p} < y_{p}^*$ and $y_{q} > y_{q}^*$.
Note that $\bm{y}^*$ has the form $y_{j}^* = \min\{\max\{c(k),t_{j}\},u_{j}\}$ for all~$j$.
It follows from $t_{p} \le y_{p} < y_{p}^*$ that $y_{p}^* \le c(k)$.
Similarly, it follows from $y_{q}^* < y_{q} \le u_{q}$ that $y_{q}^* \ge c(k)$.
Therefore, $y_{p} < y_{p}^* \le c(k) \le y_{q}^* < y_{q}$.
Choose $\varepsilon = \min\{y_{p}^*-y_{p},\; y_{q}-y_{q}^*\} > 0$,
and consider $\bm{y}'$ given by $y'_{p} = y_{p} + \varepsilon$, \, $y'_{q} = y_{q} - \varepsilon$, \, and $y'_{j} = y_{j}$ for all $j \notin \{p,q\}$.
Then $t_{p} \le y_{p} < y'_{p} \le y_{p}^* \le u_{p}$, \, $t_{q} \le y_{q}^* \le y'_{q} < y_{q} \le u_{q}$, \, and $\sum_{j=1}^{d} y'_{j} = \sum_{j=1}^{d} y_{j}$, and thus $\bm{y}' \in \mathcal{Y}_{k}$.
Note that $y_{p} < y'_{p} \le y_{p}^* \le c(k) \le y_{q}^* \le y'_{q} < y_{q}$, and thus $\bm{y}' \preceq \bm{y}$.

It follows from the choice of $\varepsilon$ that $y'_{p} = y_{p}^*$ or $y'_{q} = y_{q}^*$.
Repeating this exchange finitely many times yields $\bm{y}^*$ from $\bm{y}$, and transitivity gives $\bm{y}^* \preceq \bm{y}$.
This establishes \eqref{cl:ordering}.
\end{proof}

\newcommand{\waterstateKTwo}[2]{%
\begin{tikzpicture}[x=0.93cm,y=1.10cm,baseline=(current bounding box.north)]
    \node[font=\small] at (2.25,3.70) {#1};
    \foreach \t/\ell/\u/\hascap/\caplabel [count=\i] in {#2} {
        \pgfmathsetmacro{\x}{\i-1}
        \draw[thick] (\x,0) rectangle (\x+0.68,3.35);
        \fill[gray!30] (\x,0) rectangle (\x+0.68,\t);
        \fill[cyan!35!newblue!45] (\x,\t) rectangle (\x+0.68,\ell);
\ifnum\hascap=1
    \draw[neworange,dashed,very thick] (\x,\u) -- (\x+0.68,\u);

    \if\relax\detokenize{\caplabel}\relax
    \else
        \node[
            font=\scriptsize,
            anchor=south,
            neworange!70!black,
            fill=white,
            fill opacity=0.9,
            text opacity=1,
            inner sep=1pt,
            yshift=3pt
        ] at (\x+0.34,\u) {$\caplabel$};
    \fi
\fi
        \draw[cyan!35!newblue!45,thin] (\x,\ell) -- (\x+0.68,\ell);
        \node[below,font=\scriptsize] at (\x+0.34,0) {$\i$};
    }
    \draw[->] (-0.38,0) -- (-0.38,3.50);
    \node[rotate=90,font=\scriptsize] at (-0.75,1.75) {level};
\end{tikzpicture}%
}

\begin{figure}[t]
\centering

\begin{subfigure}{0.47\textwidth}
\centering
\waterstateKTwo{Initial}{
1.0/1.0/1.1/1/{u_1=t_3},
1.1/1.1/1.3/1/{u_2=t_4},
1.1/1.1/3.0/1/{u_3=t_5},
1.3/1.3/3.3/0/{},
3.0/3.0/3.3/0/{} }
\caption{Initial levels $\bm t=(1.0,1.1,1.1,1.3,3.0)^\top$.}
\end{subfigure}
\hfill
\begin{subfigure}{0.47\textwidth}
\centering
\waterstateKTwo{After $s=0.1$}{
1.0/1.1/1.1/1/{},
1.1/1.1/1.3/1/{},
1.1/1.1/3.0/1/{},
1.3/1.3/3.3/0/{},
3.0/3.0/3.3/0/{} }
\caption{$c=1.1$ and bucket $1$ reaches its cap $u_1=t_3$.}
\end{subfigure}

\vspace{1ex}

\begin{subfigure}{0.47\textwidth}
\centering
\waterstateKTwo{After $s=0.5$}{
1.0/1.1/1.1/1/{},
1.1/1.3/1.3/1/{},
1.1/1.3/3.0/1/{},
1.3/1.3/3.3/0/{},
3.0/3.0/3.3/0/{} }
\caption{$c=1.3$ and bucket $2$ reaches its cap $u_2=t_4$.}
\end{subfigure}
\hfill
\begin{subfigure}{0.47\textwidth}
\centering
\waterstateKTwo{Final: $s=2$}{
1.0/1.1/1.1/1/{},
1.1/1.3/1.3/1/{},
1.1/2.05/3.0/1/{},
1.3/2.05/3.3/0/{},
3.0/3.0/3.3/0/{} }
\caption{$c(2)=2.05$ and $\bm\beta^*=(0.1,0.2,0.95,0.75,0)^\top$.}
\end{subfigure}

\vspace{1ex}

\begin{tikzpicture}
    \fill[gray!30] (0,0) rectangle (0.4,0.25);
    \node[right] at (0.4,0.125) {\scriptsize initial level $t_j$};

    \fill[cyan!35!newblue!45] (3.1,0) rectangle (3.5,0.25);
    \node[right] at (3.5,0.125) {\scriptsize added water $\beta_j$};

    \draw[neworange,dashed,very thick] (6.0,0.125) -- (6.5,0.125);
    \node[right] at (6.5,0.125) {\scriptsize Weyl cap $u_j=t_{j+2}$};
\end{tikzpicture}

\caption{
Water filling for \Cref{example1}. For $k=2$, the finite upper bound for
bucket $j$ is the higher eigenvalue $u_j=t_{j+2}$. Thus the first three buckets
have caps $u_1=t_3=1.1$, $u_2=t_4=1.3$, and $u_3=t_5=3.0$, while the last two
buckets have no finite Weyl cap. Here, we let $s$ denote the amount of water usage, subject to the total water budget of $2$.
}
\label{fig:water-filling-k2}
\end{figure}
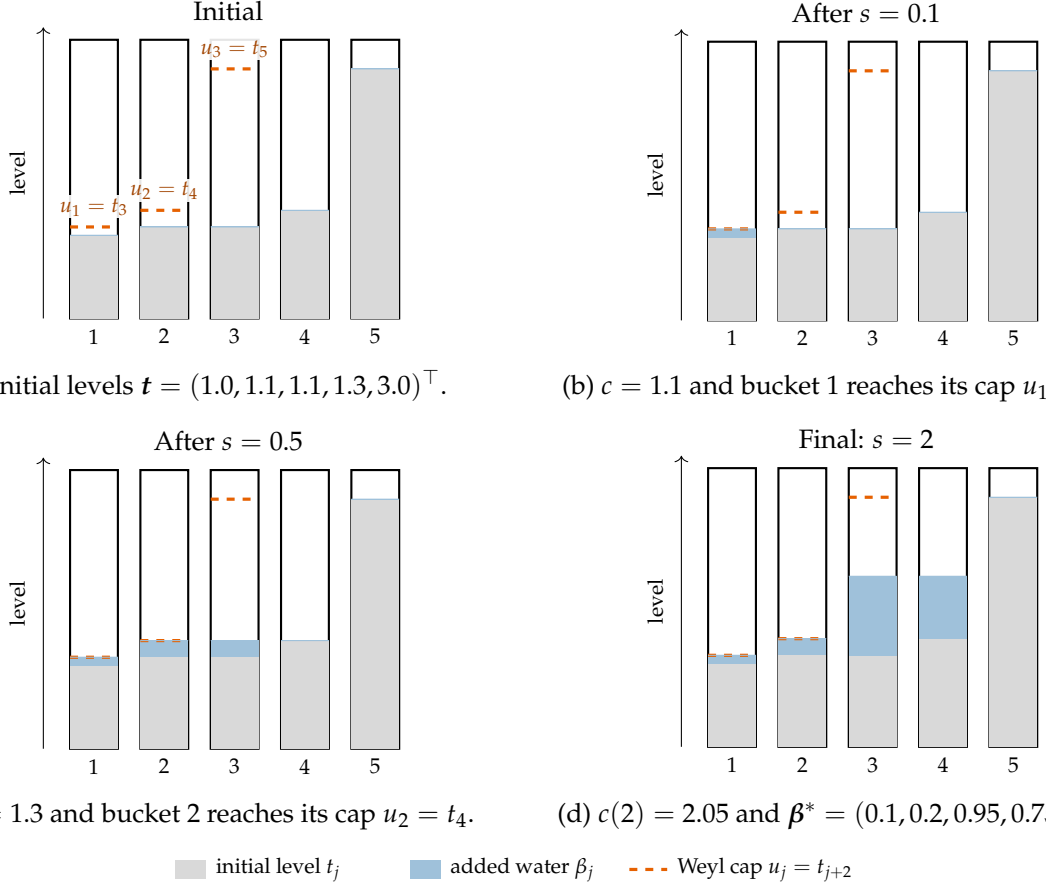

The following example illustrates the result of \Cref{prop:water-filling-monotone}.
\begin{example}
\label{example1}

We consider a symmetric, convex, 
nonincreasing function $f$. Let $k=2$, $d=5$, and
$
    \bm t=(1.0,1.1,1.1,1.3,3.0)^\top.
$
Then $\hat d=2$ and
\[
    \bm u=(t_3,t_4,t_5,\infty,\infty)^\top
    =(1.1,1.3,3.0,\infty,\infty)^\top.
\]
Thus the first three buckets have finite Weyl caps
$u_1=t_3=1.1$, $u_2=t_4=1.3$, and $u_3=t_5=3.0$, while the last two buckets
have an infinite cap. This example is illustrated in \Cref{fig:water-filling-k2}.

For this example, the water increments, as a function of the water level $c$, are given by
\[
\begin{aligned}
\beta_1&=
\begin{cases}
 c-1.0, & c\in[1.0,1.1],\\
 0.1, & c\in[1.1,\infty),
\end{cases}
&
\beta_2&=
\begin{cases}
0, & c\in[1.0,1.1],\\
c-1.1, & c\in[1.1,1.3],\\
0.2, & c\in[1.3,\infty),
\end{cases}\\[2mm]
\beta_3&=
\begin{cases}
0, & c\in[1.0,1.1],\\
c-1.1, & c\in[1.1,3.0],\\
1.9, & c\in[3.0,\infty),
\end{cases}
&
\beta_4&=
\begin{cases}
0, & c\in[1.0,1.3],\\
c-1.3, & c\in[1.3,\infty),
\end{cases}\\[2mm]
\beta_5&=
\begin{cases}
0, & c\in[1.0,3.0],\\
c-3.0, & c\in[3.0,\infty).
\end{cases}
\end{aligned}
\]
Thus, the total amount of water used by level $c$ is
\[
    s\defi\sum_{j=1}^5 \beta_j
    =
    \begin{cases}
    c-1.0, & c\in[1.0,1.1],\\
    2c-2.1, & c\in[1.1,\infty),
    \end{cases}
\]
Hence, for the budget $k=2$, the final water level $c$ is given by $c=2.05,$
and the water-filling solution is
\[
    \bm\beta^*=(0.1,0.2,0.95,0.75,0)^\top,
    \qquad
    \bm t+\bm\beta^*=(1.1,1.3,2.05,2.05,3.0)^\top.
\]
\end{example}

\subsection{An Optimal Solution to the Lower Bound Problem: General Case}

Now, we consider the case when $f$ is not necessarily non-increasing.
In this case, the optimal $\bm{\beta}$ is again given by the same water-filling argument, but it may use a total budget of $s \leq k$.
We generalize the definitions from the non-increasing case to allow a chosen amount $s$ of water to be filled, not necessarily equal to $k$.
For any $s \geq 0$, let
\begin{equation}
\label{eqn:ck-definition2}
c(s) \ \ \defi \ \ \inf\{c \; : \; c \ge t_{1}, \; \Phi_{k}(c) \ge s\}.
\end{equation}
In words, $c(s)$ is the water level reached after pouring $s$ units of water, in lockstep, into the buckets with the lowest initial water levels.
Let
\begin{equation}
\label{eqn:beta-water-filling2}
\beta_{j}(s) \ \ \defi \ \ \min\big\{(c(s) - t_{j})_+, \, u_{j} - t_{j}\big\},
\qquad j=1,\ldots,d.
\end{equation}
That is, the vector $\bm{\beta}(s)$ is obtained by raising the smallest non-full buckets in lockstep until the budget used is exactly $s$.

\begin{lemma}
\label{lem:general}
For any symmetric convex $f : \mathbb{R}^{d} \rightarrow \mathbb{R} \cup \{\infty\}$, and non-negative reals $t_{1},\ldots,t_{d}$ such that $0 \leq t_{1} \leq \cdots \leq t_{d}$, if the problem~\eqref{eqn:relaxation-general} has an optimal solution, then there exists $0 \leq s^* \leq k$ such that $\bm{\beta}(s^*)$ as defined in \eqref{eqn:beta-water-filling2} is an optimal solution to \eqref{eqn:relaxation-general}.
\end{lemma}
\begin{proof}
Let $\bm{\hat{\beta}}$ be an optimal solution to \eqref{eqn:relaxation-general}. Let $s^*=\sum_{i=1}^{d} \hat{\beta_{i}}$. Then consider the following optimization problem where we restrict the total budget to exactly $s^*$.
\begin{equation}
\label{eqn:restricted-general}
\min_{\bm{\beta} \in \mathbb{R}_+^{d}} \left\{f(t_{1} + \beta_{1},\ldots,t_{d} + \beta_{d})   : t_{j} + \beta_{j} \le t_{j+\hat{d}} \  \forall \, j=1,\ldots,d-\hat{d}, \sum_{j=1}^{d} \beta_{j} =s^*\right\}.
\tag{Restricted-Eigen}
\end{equation}

Since $\hat{\bm{\beta}}$ is feasible  and optimal for
problem~\eqref{eqn:relaxation-general}, the optimal value of
problem~\eqref{eqn:restricted-general} is no larger than the optimal
value of problem~\eqref{eqn:relaxation-general}.

Now,  as in the proof of \Cref{prop:water-filling-monotone}, we let 
\[
\mathcal{Y}_{s^*}
\defi
\left\{
\bm{y}\in\mathbb{R}^{d}:
 t_{j}\le y_{j}\le u_{j}\ \forall j = 1, \ldots, d,
 \sum_{j=1}^{d} y_{j}=\sum_{j=1}^{d} t_{j}+s^*
\right\},
\]
and let $\bm{\beta}(s^*)$ denote the solution output by the water-filling algorithm described in \Cref{sec_opt_lb:monotone} with the same budget as $\bm{\hat{y}}$. Then let $\bm{y}^*\defi\bm{t}+\bm{\beta}(s^*)$ and $\bm{\hat{y}}\defi\bm{t}+\bm{\hat{\beta}}$. Clearly, both $\bm{y}^*,\bm{\hat{y}}\in \mathcal{Y}_{s^*}$. Moreover, as in proof of \Cref{prop:water-filling-monotone},  $\bm{y^*}\preceq \bm{\hat{y}}$ since $\bm{y^*}$ is obtained by water-filling algorithm with the same budget as $\bm{\hat{y}}$. Now, from \Cref{lem:majorization-schur-convexity}, we have $f(\bm{y}^*)\leq f(\bm{\hat{y}})$ proving optimality of $\bm{y}^*$ and thus of $\bm{\beta}(s^*).$
\end{proof}

Now, let $g:\mathbb{R}_+\rightarrow \mathbb{R}\cup \{\infty\}$ be defined as $$g(s):=\min_{\bm{\beta} \in \mathbb{R}_+^{d}} \left\{f(t_{1} + \beta_{1},\ldots,t_{d} + \beta_{d}) \  : \  t_{j} + \beta_{j} \le t_{j+\hat{d}} \  \forall \, j=1,\ldots,d-\hat{d}, \ \sum_{j=1}^{d} \beta_{j} =s\right\}.$$
A simple observation shows that $g$ is convex (see, e.g., \cite[Sec.~3.2.5]{boyd2004convex}). Thus, we can employ a bisection or golden section search to efficiently compute the optimal $s^*$ whose existence is shown in \Cref{lem:general}. 
We remark that for any fixed $s$, the computation of $g(s)$ and corresponding $\bm{\beta}(s)$ is done via the water-filling algorithm. Thus, given the sorted eigenvalues $t_{1},\ldots, t_{d}$ of $\bm{A}$, we can obtain $0\leq \hat{{s}}\leq k$ such that $|\hat{{s}}-{s}^*|\leq \epsilon/L$ in $\log(k L/\epsilon)$ iterations where $L$ is the Lipschitz constant of $f$. In each iteration, we compute $\bm{\beta}(s)$ and evaluate $f(\bm{\beta}(s)+\bm{t})$ at the current value of $s$ in the bisection search which can be done in $O(d)$ time. Finally, $f(\bm{\beta(\hat{s}})+\bm{t})\leq f(\bm{\beta}(s^*)+\bm{t})+\epsilon$ as claimed.

Finally, we present sufficient conditions under which the problem~\eqref{eqn:relaxation-general} has an optimal solution and show that there is always an optimal solution under mild conditions on $f$ as long as $k$ is at least the dimension of the null-space of $A$, 
that is, the number of zero elements in $\bm{t}$.
Let $\|\bm{t}\|_{0}$ denote the support size of vector $\bm{t}$.

\begin{proposition}
\label{prop:finiteness_pd_domain}
Suppose that $f$ is symmetric,
lower semicontinuous, convex, and finite on $\mathbb{R}^{d}_{++}$. 
If
$k \ge d-\|\bm{t}\|_{0}$,  then the problem~\eqref{eqn:relaxation-general} has a finite optimal value and it is attained.
\end{proposition}

\begin{proof}
Let $r\coloneqq d-\|\bm{t}\|_{0}$. Since $0\le t_{1}\le \cdots \le t_{d}$, we have $t_{1}=\cdots=t_{r}=0$ and $t_{r+1}>0,\ldots,t_{d}>0$ when $r\ge1$.

Suppose that $k\ge r$. If $r=0$, then $\bm{t}\in\mathbb{R}_{++}^{d}$, and $\bm{\beta}=\bm{0}$ is feasible with finite objective value. Now suppose $r\ge1$. Let $\hat{d}=\min\{d,k\}$. Since $k\ge r$, we have $\hat{d}\ge r$. We construct a feasible $\bm{\beta}$ such that $\bm{t}+\bm{\beta}\in\mathbb{R}_{++}^{d}$.

Choose $\epsilon>0$ small enough so that
$
r\epsilon\le k
$
and, for every $j\in\{1,\ldots,r\}\cap\{1,\ldots,d-\hat{d}\}$,
$
\epsilon\le t_{j+\hat{d}}.
$
Such an $\epsilon$ exists because $j+\hat{d}\ge j+r>r$ for all $j=1,\ldots,r$ whenever $j\le d-\hat{d}$, and hence $t_{j+\hat{d}}>0$. Define
$
\beta_{j} =
\begin{cases}
\epsilon, & j=1,\ldots,r,\\
0, & j=r+1,\ldots,d.
\end{cases}
$
Then $\bm{\beta}\in\mathbb{R}_+^{d}$ and $\sum_{j=1}^{d}\beta_{j}=r\epsilon\le k$. Moreover, for $j=1,\ldots,d-\hat{d}$, the constraint $t_{j}+\beta_{j}\le t_{j+\hat{d}}$ holds. Indeed, if $j\le r$, then $t_{j}=0$ and $\beta_{j}=\epsilon\le t_{j+\hat{d}}$ by the choice of $\epsilon$; if $j>r$, then $\beta_{j}=0$ and $t_{j}\le t_{j+\hat{d}}$ because $\bm{t}$ is non-decreasing. Therefore, $\bm{\beta}$ is feasible for~\eqref{eqn:relaxation-general}. Since $\bm{t}+\bm{\beta}\in\mathbb{R}_{++}^{d}$, the objective value $f(\bm{t}+\bm{\beta})$ is finite.

It remains to show attainment. The feasible set of~\eqref{eqn:relaxation-general} is closed and bounded because $\bm{\beta}\in\mathbb{R}_+^{d}$ and $\sum_{j=1}^{d}\beta_{j}\le k$; hence it is compact. Since $f$ is lower semicontinuous on the compact feasible set, the objective $\bm{\beta}\mapsto f(\bm{t}+\bm{\beta})$ is lower semicontinuous on the compact feasible set. Since there is at least one feasible point with finite objective value, the Weierstrass theorem for lower semicontinuous functions implies that the infimum is attained by some feasible $\bm{\beta}^*$. Thus, problem~\eqref{eqn:relaxation-general} has a finite optimal value and an optimal solution.
\end{proof}

The condition $k\ge d-\|\bm t\|_0$ is not only sufficient for many classical design criteria, but also necessary. Suppose, for example, that $f$ is finite on $\mathbb{R}_{++}^{d}$ and infinite on $\mathbb{R}_{+}^{d}\setminus \mathbb{R}_{++}^{d}$, as in A-, D-, and E-optimality. If $k<r\coloneqq d-\|\bm t\|_0$, then $\hat d=k$ and, for every $j=1,\ldots,r-k$, we have $t_j=0$ and $t_{j+\hat d}=t_{j+k}=0$. The constraint $t_j+\beta_j\le t_{j+\hat d}$ therefore forces $\beta_j=0$ for $j=1,\ldots,r-k$. Hence $\bm t+\bm\beta\notin\mathbb{R}_{++}^{d}$ for every feasible $\bm\beta$, and the objective value is infinite for every feasible solution. Thus, in this case, $k\ge d-\|\bm t\|_0$ is necessary and sufficient for problem~\eqref{eqn:relaxation-general} to have a finite optimal value. Equivalently, in the original matrix formulation, if $\dim(\ker(\bm A))>k$, then no rank-$k$ update $\bm X\bm X^\top$ can make $\bm A+\bm X\bm X^\top$ positive definite. Indeed, there exists a nonzero vector in $\ker(\bm A)\cap\ker(\bm X^\top)$, so $\bm A+\bm X\bm X^\top$ remains singular.

\section{Construction of an Optimal Design Matrix}

\Cref{prop:lower_bound_general} implies that for any prior information matrix $\bm{A} = \operatorname{diag}(t_{1}, t_{2},\ldots, t_{d})$ such that $0 \leq t_{1} \leq t_{2} \leq \cdots \leq t_{d}$, and a design matrix $\bm{X} = [\bm{x}^{1},\ldots,\bm{x}^{k}]$ such that $\bm{x}^{i} \in B(\bm{0},1)$ for all $i = 1,\ldots,k$, it holds that
\[
F\left(\bm{A} + \bm{X} \bm{X}^{\top}\right) \geq f(\bm{t} + \bm{\beta}(s^*)),
\]
where $\bm{\beta}(s^*)$ is an optimal solution of problem \eqref{eqn:relaxation-general} provided by \Cref{lem:general} for some $0\leq s^*\leq k$. 
In this section, we show how to find $k$~vectors $\bm{x}^{1}, \ldots, \bm{x}^{k} \in B(\bm{0},1)$ such that $F\left(\bm{A} + \bm{X} \bm{X}^{\top}\right) = f(\bm{t} + \bm{\beta}(s^*))$. The most natural approach is to aim for vectors $\bm{x}^{1}, \ldots, \bm{x}^{k} \in B(\bm{0},1)$ such that
\begin{equation}\label{eqn:construction_claim}
\bm{X} \bm{X}^{\top} = \operatorname{diag}\left(\beta_{1}(s^*),\dots,\beta_{d}(s^*)\right).
\end{equation}
which would imply that $\bm{A} + \bm{X} \bm{X}^{\top} = \operatorname{diag}(t_{1}, t_{2},\ldots, t_{d})+\operatorname{diag}(\beta_{1}(s^*), \beta_{2}(s^*),\ldots, \beta_{d}(s^*))=\operatorname{diag}(t_{1}+\beta_{1}(s^*),\ldots, t_{d}+\beta_{d}(s^*))$ and therefore, $F\left(\bm{A} + \bm{X} \bm{X}^{\top}\right) = f(\bm{t} + \bm{\beta}(s^*)).$

In the following example, we show such vectors $\bm{x}^{1},\ldots, \bm{x}^{k}$ may not exist. 
\begin{example}
\label{example3}
Assume that $f$ is non-increasing and consider the optimal solution $\bm{\beta}^* = \bm{\beta}(2) = (0.1, 0.2, 0.95, 0.75, 0.0)^{\top}$ of \Cref{example1} with $s=k=2$.
Note that $\operatorname{rank}(\operatorname{diag}(\bm{\beta}(2))) = 4 > k \ge \operatorname{rank}\left(\bm{X} \bm{X}^{\top}\right)$, and thus there does not exist a design matrix $\bm{X} = [\bm{x}^{1},\ldots,\bm{x}^{k}]$ such that $\bm{X} \bm{X}^{\top} = \operatorname{diag}(\bm{\beta}^*)$.
\end{example}

Instead, we show the following crucial lemma which shows that there is always an alternate optimal solution $\bm{\beta}'(s)$ whose support is at most $\hat{d}=\min\{d,k\}$.

\begin{lemma}
\label{lem:simp_alpha}
For any $s \in [0,k]$, let $\bm{\beta}(s)$ denote the solution defined by the water-filling algorithm with budget $s$ as in \eqref{eqn:beta-water-filling2}. 
Define $\bm{\beta}'(s)\in\mathbb{R}^{d}$ by
\[
\beta'_{j}(s)
\defi
\begin{cases}
(c(s)-t_{j})_+, & j=1,\ldots,\hat{d},\\
0, & j=\hat{d}+1,\ldots,d.
\end{cases}
\]
Then  $\bm{\beta}'(s)\in \mathbb{R}^{d}$ has support at most $\min\{d,k\}$, and
there exists a permutation matrix $\bm{P}$ such that $\bm{P}(\bm{t}+\bm{\beta}(s))=\bm{t}+\bm{\beta}'(s)$. 
Consequently, for every symmetric function $f$,
$f(\bm{t} + \bm{\beta}(s)) = f(\bm{t} + \bm{\beta}'(s))$.
\end{lemma}

We give a proof of \Cref{lem:simp_alpha} in \Cref{sec:simp_beta}. Here we revisit \Cref{example3} and illustrate the vector $\bm{\beta}'(s)$ for this example.
\begin{example}
    Continuing \Cref{example1}, recall that $\bm{t} = (1.0, 1.1, 1.1, 1.3, 3.0)^{\top}$, note that $\bm{t} + \bm{\beta}(2) = (1.1, 1.3, 2.05, 2.05, 3.0)^{\top}$, and thus $\bm{t} + \bm{\beta}(2)$ contains only $k = 2$ values not in $\bm{t}$.
Except for the change of order, $\bm{t} + \bm{\beta}(2)$ has the same values as $\bm{t} + \bm{\beta}'(2)$ where $\bm{\beta}'(2) = (1.05,0.95,0.0,0.0,0.0)^{\top}$.
It follows from the symmetry of $f$ that $f(\bm{t} + \bm{\beta}'(2)) = f(\bm{t} + \bm{\beta}(2))$.
\end{example}

The next step is to construct vectors  $\bm{x}^{1}, \ldots, \bm{x}^{k} \in B(\bm{0},1)$ such that
\begin{equation*}
\bm{X} \bm{X}^{\top} = \operatorname{diag}(\beta_{1}'(s^*),\dots,\beta_{d}'(s^*)),
\end{equation*}
which would imply that $\bm{A} + \bm{X} \bm{X}^{\top} =\operatorname{diag}(t_{1}+\beta'_{1}(s^*),\ldots, t_{d}+\beta'_{d}(s^*))$ and therefore, $F\left(\bm{A} + \bm{X} \bm{X}^{\top}\right) = f(\bm{t} + \bm{\beta}(s^*)).$

We present the following structural lemma. It characterizes exactly which PSD matrices can be written as a sum of $k$ rank-one matrices generated by vectors in the unit ball. The proof uses the Schur--Horn theorem and also implies \Cref{cor:spectral-design-rank}, which is delayed in \Cref{sec:construction}.
\begin{lemma}
\label{lem:construction}
The feasible set of problem~\eqref{eq:SpectralDesignM} is
\[
\left\{
\bm{M} = \sum_{i=1}^{k} \bm{x}^{i} \left(\bm{x}^{i}\right)^{\top}
: 
\bm{x}^{1},\ldots,\bm{x}^{k} \in B(\bm{0},1)
\right\}
=
\Big\{
\bm{M} \in \mathbb{S}_{d}^+ :
\operatorname{Tr}(\bm{M}) \le k,\;
\operatorname{rank}(\bm{M}) \le k
\Big\}.
\]
\end{lemma}

As a consequence of the proof of \Cref{lem:construction}, the representation of the left-hand side can be constructed algorithmically.
\begin{corollary}
\label{cor:construction_algorithm}
Given a matrix $\bm{M} \in \mathbb{S}_{d}^+$ with $\operatorname{Tr}(\bm{M}) \le k$ and $\operatorname{rank}(\bm{M}) \le k$, one can construct vectors $\bm{x}^{i} \in B(\bm{0},1)$, $i=1,\ldots,k$, such that $\bm{M} = \sum_{i=1}^{k} \bm{x}^{i} \left(\bm{x}^{i}\right)^{\top} $
in runtime $O(d^{3}+kd)$.
\end{corollary}

Observe that
\[
\operatorname{diag}\left(\bm{\beta}'(s^*)\right)
\in
\Big\{
\bm{M} \in \mathbb{S}_{d}^+ :
\operatorname{Tr}(\bm{M}) \le k,\;
\operatorname{rank}(\bm{M}) \le k
\Big\}.
\]
Thus, \Cref{lem:construction} guarantees the existence of design vectors attaining this matrix, and \Cref{cor:construction_algorithm} provides a constructive procedure for computing them. Therefore, to complete the proof of \Cref{thm:general_opt_design}, it remains to prove \Cref{lem:simp_alpha,lem:construction} and \Cref{cor:construction_algorithm}, which we do in the following subsections.

\subsection{An Alternative Optimal Solution with Small Support}\label{sec:simp_beta}

In this section, we prove \Cref{lem:simp_alpha} and show another solution $\bm{\beta}'$ with support at most $k$ such that $P(\bm{t}+\bm{\beta}(s))=\bm{t}+\bm{\beta}'(s)$. 

\begin{proof}[{Proof of \Cref{lem:simp_alpha}}]
 If $k\ge d$, then $\hat{d}=d$, the upper-bound constraints in
problem~\eqref{eqn:relaxation-general} are vacuous,
and $\bm{\beta}'(s)=\bm{\beta}(s)$; hence $\bm{P}=\bm I_{d}$ works.
Thus, assume $k<d$.

Next we show that $\bm{t} + \bm{\beta}(s)$, as a multiset, contains at most $k$ values not in $\bm{t}$. Recall that $c(s)= \inf\{c\ge t_{1}:\Phi_{k}(c)\ge s\}$ and $\beta_{j}(s)
=
\min\big\{(c(s)-t_{j})_+,\,u_{j}-t_{j}\big\}$  for each $j=1,\ldots,d$ where $u_{j}=t_{j+k}$ if $j\leq d-k$ and $+\infty$ otherwise. Thus $t_{j}+\beta_{j}(s)= \min\big\{(c(s)-t_{j})_++t_{j},u_{j}\big\}$. Thus for any $1\leq j\leq d$, we have 
\begin{align}\label{eq_cases}
t_{j} + \beta_{j}(s) = \begin{cases}
t_{j + k} & \text{ if } t_{j+k}\leq c(s), \\
c(s) & \text{ if } t_{j}<c(s)< t_{j+k}, \\
t_{j} & \text{ if } t_{j}\geq c(s),
\end{cases}
\end{align}
where we use the convention that $t_{j} \defi \infty$ if $j > d$.

Now we claim that the second case in \eqref{eq_cases} only happens for at most $k$ indices. Indeed if $t_{j}<c(s)< t_{j+k}$ then clearly the index $j+k$ satisfies $t_{j+k}> c(s)$ and thus $t_{j+k}+\beta_{j+k}(s)$ is set according to the last case and equals $t_{j+k}$. Thus, suppose that $j^*$ denotes the smallest index such that $t_{j^*}+\beta_{j^*}(s)=c(s)$. Then, we must have $t_{j^*+k}\geq c(s)$. Thus, there are at most $k$ indices $j^*,j^*+1,\ldots,j^*+k-1$ in the second case of \eqref{eq_cases}.

Next, we discuss two subcases. 
\par\noindent \textbf{Subcase I.} First, there is no $j$ that satisfies the first case in \eqref{eq_cases}, i.e., there exists no $j$ such that $t_{j}+\beta_{j}(s)=t_{j+k}$. In this case, observe that $t_{j}+\beta_{j}(s)=c(s)$ for $1\leq j\leq r$ for some $r\leq k$ and $t_{j}+\beta_{j}(s)=t_{j}$ for $j>r$. We can define $\bm{\beta}'(s)=\bm{\beta}(s)$ and we are done. 

\par\noindent \textbf{Subcase II.} Else, let $\ell\geq 1$ be the largest index such that $t_{\ell+k}\leq c(s)$. Thus $c(s)< t_{\ell+k+1}$. Since we assume $k<d$, we must have $\ell\leq d-k$.

Then for each $1\leq j\leq \ell$, we have $t_{j+k}\leq t_{\ell+k}\leq c(s)$ and thus 
$$ t_{j} + \beta_{j}(s)= t_{j+k}.$$

Now for each $\ell+1\leq j\leq \ell+k$, we either have $t_{j}= t_{\ell+k}= c(s)$ or else have $t_{j}\leq t_{\ell+k}\leq c(s)$, $t_{j}<c(s)$ , and $t_{j+k}\geq t_{\ell+k+1}> c(s)$. Thus, in either way, we must have 
$$t_{j}+\beta_{j}(s)=c(s).$$

Now, if $j\geq \ell+k+1$, then we have $t_{j}\geq t_{\ell+k+1}> c(s)$ and thus we have 
$t_{j}+\beta_{j}(s)=t_{j}$.

Thus, as a multiset,
\[
\{t_{1}+\beta_{1}(s),\ldots,t_{d}+\beta_{d}(s)\}
=
\{t_{k+1},\ldots,t_{\ell+k},
\underbrace{c(s),\ldots,c(s)}_{k\text{ times}},
 t_{\ell+k+1},\ldots,t_{d}\}.
\]
Equivalently,
\[
\{t_{1}+\beta_{1}(s),\ldots,t_{d}+\beta_{d}(s)\}
=
\{\underbrace{c(s),\ldots,c(s)}_{k\text{ times}},t_{k+1},\ldots,t_{d}\}.
\]
Define $\bm{\beta}'(s)\in\mathbb{R}_+^{d}$ by
\[
    \beta'_{j}(s)
    \defi
    \begin{cases}
        c(s)-t_{j}, & j=1,\ldots,k,\\
        0, & j>k.
    \end{cases}
\]
Since $t_{k} \le t_{\ell+k} \leq c(s)$, we have $\bm{\beta}'(s)\ge0$, and
$|\operatorname{supp}(\bm{\beta}'(s))|\le k$. Moreover, the entries of
$\bm{t}+\bm{\beta}'(s)$ are exactly
\[
\{\underbrace{c(s),\ldots,c(s)}_{k\text{ times}},t_{k+1},\ldots,t_{d}\},
\]
which is the same multiset as the entries of $\bm{t}+\bm{\beta}(s)$. Therefore,
there exists a permutation matrix $\bm{P}$ such that
\[
    \bm{P}(\bm{t}+\bm{\beta}(s))=\bm{t}+\bm{\beta}'(s).
\]
Since $f$ is symmetric,
\[
    f(\bm{t}+\bm{\beta}(s))
    =
    f(\bm{P}(\bm{t}+\bm{\beta}(s)))
    =
    f(\bm{t}+\bm{\beta}'(s)).
\]
This completes the proof.
\end{proof}

\subsection{A Feasible Design Achieving the Lower Bound}\label{sec:construction}

In this section, we prove \Cref{lem:construction} and \Cref{cor:construction_algorithm}. Clearly, the containment ``$\subseteq $'' is straightforward. We now show ``$\supseteq$'' and then give the algorithm as claimed.

We use the following Schur--Horn theorem.
\begin{theorem}[{Schur--Horn theorem, see \cite[Theorems~4.3.45 and~4.3.48]{horn2012matrix}}]
Let $\bm{Y} \in \mathbb{R}^{n \times n}$ be a real symmetric matrix with eigenvalues $\lambda_{1} \ge \cdots \ge \lambda_{n}$ and diagonal entries $\bm{y}$ indexed such that $y_{1} \geq \cdots \geq y_{n}$.
Then $\bm{y} \preceq \bm{\lambda}$.
Conversely, if $\bm{y} \in \mathbb{R}^{n}$, indexed such that $y_{1} \geq \cdots \geq y_{n}$, and $\bm{\lambda} \in \mathbb{R}^{n}$, indexed such that $\lambda_{1} \ge \cdots \ge \lambda_{n}$, satisfies $\bm{y} \preceq \bm{\lambda}$, then there exists a real symmetric matrix $\bm{Y}$ with eigenvalues $\bm{\lambda}$ and diagonal entries $\bm{y}$.
\end{theorem}

\begin{proof}[{Proof of \Cref{lem:construction}}]
Suppose that $\bm{M}\in\mathbb{S}_{d}^+$ satisfies $\operatorname{Tr}(\bm{M})\le k$ and $\operatorname{rank}(\bm{M})\le k$. Let $r=\operatorname{rank}(\bm{M})$ and let
$
\bm{M}=\sum_{j=1}^{r}\lambda_{j} \bm{w}^{j}(\bm{w}^{j})^{\top}
$
be an eigenvalue decomposition of $\bm{M}$, where $\lambda_{j}>0$ for $j=1,\ldots,r$ such that $\lambda_{1}\geq \cdots\geq \lambda_{r}$ and $\bm{w}^{1},\ldots,\bm{w}^{r}$ are orthonormal. Let $s=\operatorname{Tr}(\bm{M})=\sum_{j=1}^{r}\lambda_{j}$.

If $s=0$, then $\bm{M}=\bm{0}$ and the desired representation is obtained by taking $\bm{x}^{i}=\bm{0}$ for all $i=1,\ldots,k$. Hence, suppose $s>0$.

Define the vector $\bm{\lambda}\in\mathbb{R}^{k}$ by padding the nonzero eigenvalues with zeros,
that is, 
$
\bm{\lambda}=(\lambda_{1},\ldots,\lambda_{r},0,\ldots,0)^{\top}$.
Let $\bm{a}=(s/k,\ldots,s/k)^{\top}\in\mathbb{R}^{k}$. Since $\bm{\lambda}$ is nonnegative and has sum $s$, the vector $\bm{a}$ is majorized by $\bm{\lambda}$. By the Schur--Horn theorem, there exists a symmetric positive semidefinite matrix $\bm{G}\in\mathbb{S}_{k}^+$ with eigenvalues $\bm{\lambda}$ and diagonal entries
$
G_{ii}=s/k$, 
$i=1,\ldots,k$.

Since $\bm{G}$ has the same nonzero eigenvalues as $\bm{M}$, there exists a matrix $\bm{X}\in\mathbb{R}^{d\times k}$ such that
$
\bm{X}\bm{X}^{\top}=\bm{M}
$
and
$
\bm{X}^{\top}\bm{X}=\bm{G}.
$
Indeed, if $\bm{G}=\bm{V}\operatorname{diag}(\bm{\lambda})\bm{V}^{\top}$ is an eigenvalue decomposition of $\bm{G}$, and $\bm{W}=[\bm{w}^{1},\ldots,\bm{w}^{r}]\in\mathbb{R}^{d\times r}$, then one may take
$
\bm{X}
=
\bm{W}
\begin{pmatrix}
\operatorname{diag}(\sqrt{\lambda_{1}},\ldots,\sqrt{\lambda_{r}}) & \bm{0}
\end{pmatrix}
\bm{V}^{\top}.
$
Then $\bm{X}\bm{X}^{\top}=\bm{M}$.

Let $\bm{x}^{i}$ denote the $i$th column of $\bm{X}$. Since $\bm{X}^{\top}\bm{X}=\bm{G}$, we have
$
\|\bm{x}^{i}\|_{2}^{2}
=
(\bm{X}^{\top}\bm{X})_{ii}
=
G_{ii}
=
s/k
\le 1
$
for all $i=1,\ldots,k$. Hence $\bm{x}^{i}\in B(\bm{0},1)$ for all $i=1,\ldots,k$, and
$
\bm{M}=\bm{X}\bm{X}^{\top}
=
\sum_{i=1}^{k}\bm{x}^{i}(\bm{x}^{i})^{\top}.
$
This proves the reverse inclusion.\end{proof}

\begin{proof}[Proof of \Cref{cor:construction_algorithm}] 
Having established \Cref{lem:construction}, 
it remains to be shown that the claimed running time holds.
Computing an eigendecomposition of $
\bm{M}=\sum_{j=1}^{r}\lambda_{j} \bm{w}^{j}(\bm{w}^{j})^{\top}
$
costs $O(d^{3})$, where 
 $r=\operatorname{rank}(\bm{M})$, 
 $\lambda_{j}>0$ for $j=1,\ldots,r$ such that $\lambda_{1}\geq \cdots\geq \lambda_{r}$ and $\bm{w}^{1},\ldots,\bm{w}^{r}$ are orthonormal
as in \Cref{lem:construction}.
We form
$
\bm{Y} = [\sqrt{\lambda_{1}}\,\bm{w}^{1},\ldots,\sqrt{\lambda_{r}}\,\bm{w}^{r},\,
\bm{0},\ldots,\bm{0}] \in \mathbb{R}^{d\times k}
$,
where $\bm{0}$ is repeated $k-r$ times; this matrix satisfies
$\bm{Y}\bm{Y}^{\top}=\bm{M}$ and
$\bm{Y}^{\top}\bm{Y}=\operatorname{diag}(\bm{\lambda})$, and is assembled
from the eigendecomposition in $O(kd)$ time. Applied to $\bm{Y}$,
Algorithm~3 of \cite{dhil:05} returns $\bm{X}\in\mathbb{R}^{d\times k}$
with $\bm{X}\bm{X}^{\top}=\bm{M}$ and $\bm{X}^{\top}\bm{X}$ having
constant diagonal $s/k$, so $\bm{G}=\bm{X}^{\top}\bm{X}$ realizes the
Schur--Horn construction of the existence argument 
in the proof of \Cref{lem:construction}. 
The runtime of
Algorithm~3 of \cite{dhil:05} is $O(kd)$ (see \cite[p.~68]{dhil:05}), so
the total runtime is $O(d^{3}+kd)$.
\end{proof}

\subsection{Computing an Optimal Spectral Design}
\label{subsec:computing-optimal-design}

In this subsection, we state the construction
in \Cref{cor:construction_algorithm}
as an explicit algorithm. The
construction works for a general positive semidefinite 
symmetric prior information matrix $\bm A$. The analysis
in \Cref{sec_opt_lb:monotone,sec:construction} applies after rotating
coordinates by an eigenspace basis of $\bm A$.
\Cref{alg:computing-optimal-design} applies under the
hypotheses of \Cref{thm:general_opt_design}.

\begin{algorithm}[t]
\begin{tcolorbox}[
    colback=white,
    colframe=black,
    boxrule=0.8pt,
    sharp corners,
    left=6pt,
    right=6pt,
    top=6pt,
    bottom=6pt
]
\caption{Optimal Spectral Design with Prior Information}
\label{alg:computing-optimal-design}
\begin{enumerate}[label=\textup{\arabic*.},nosep]

\item Compute an eigendecomposition
$
    \bm A
    =
    \bm{Q}\operatorname{diag}(t_{1},\ldots,t_{d})\bm{Q}^{\top}
$
with $0\le t_{1}\le \cdots \le t_{d}$.

\item Set $\hat{d}=\min\{d,k\}$ and define the capacities $u_{j}$ as in
\eqref{eqn:caps-water-filling}.

\item For a given $s\in[0,k]$, compute the water level $c(s)$ as in
\eqref{eqn:ck-definition2}, and define $\bm{\beta}(s)$ as in
\eqref{eqn:beta-water-filling2}. 
\item If $f$ is non-increasing, set $s^*=k$, as in
\Cref{prop:water-filling-monotone}. Otherwise, compute
an $\varepsilon$-optimal solution $s^*$ to
\[
    \min_{0\le s\le k} f(\bm{t}+\bm{\beta}(s)).
\]

\item Define $\bm{\beta}'(s^*)\in\mathbb R_+^{d}$ as in \Cref{lem:simp_alpha}.

\item Apply the  procedure from the proof of
\Cref{lem:construction} 
to
$
    \bm M=\operatorname{diag}(\bm{\beta}'(s^*))
$
to obtain vectors $\bm z^{1},\ldots,\bm z^{k}\in B(\bm 0,1)$ such that
$
    \sum_{i=1}^{k} \bm z^i(\bm z^i)^{\top}
    =
    \operatorname{diag}(\bm{\beta}'(s^*))
$.
\item Return
$
    \bm X^*=\bm{Q} \bm{Z} 
$,
where $\bm{Z} \coloneqq [\bm z^{1},\ldots,\bm z^{k}]$.
\end{enumerate}
\end{tcolorbox}
\end{algorithm}

The columns of the matrix $\bm X^*$ returned by
\Cref{alg:computing-optimal-design} belong to $B(\bm 0,1)$, because
$\bm Q$ is orthonormal and the columns of $\bm Z$ belong to
$B(\bm 0,1)$. Moreover,
\[
\bm A+\bm X^*(\bm X^*)^{\top}
=
\bm{Q}
\left(
\operatorname{diag}(\bm{t})+\operatorname{diag}(\bm{\beta}'(s^*))
\right)
\bm{Q}^{\top} .
\]
By \Cref{lem:simp_alpha}, the vector $\bm{t}+\bm{\beta}'(s^*)$ is a permutation
of $\bm{t}+\bm{\beta}(s^*)$. Since $f$ is symmetric,
the constructed design attains the lower bound in
\Cref{prop:lower_bound_general}, and is optimal for
problem~\eqref{eq:SpectralDesign}
if $f$ is non-increasing or $\varepsilon$-optimal
otherwise.

\subsection{Closed-form Optimal Designs with Isotropic Prior}

We next show that when the prior information matrix is isotropic,
i.e., $\bm{A}=\ell\bm I$ with $\ell\ge0$, the optimal design obtained in
\Cref{thm:general_opt_design} admits a simple closed-form construction.
In this case, all eigenvalues of $\bm{A}$ are equal, and the relaxation
\eqref{eqn:relaxation-general} allocates the increase in eigenvalues
uniformly across coordinates.

The following proposition constructs explicit design vectors that realize
these eigenvalue increments.
\begin{proposition}[Closed-form optimal designs]
\label{prop:closed_form_design}
Suppose that $\bm{A}=\ell\bm I_{d}$ with $\ell\ge0$. Let $s \in [0,k]$.
Then the following holds:
\begin{enumerate}[nosep]
\item[\textup{(i)}] If $k\le d$, define $\bm{x}^{i} = \sqrt{\frac{s}{k}}\,\bm e_{i}$ for each $i=1,\ldots,k.$
Then
\[
\sum_{i=1}^{k} \bm{x}^{i} (\bm{x}^{i})^{\top}
=
\operatorname{diag}\left(\tfrac{s}{k},\ldots,\tfrac{s}{k},0,\ldots,0\right),
\]
where $s/k$ is repeated $k$ times.
\item[\textup{(ii)}] If $k\ge d+1$, let $\theta_{i} = 2\pi(i-1)/k$ for $i=1,\ldots,k$ and define
\[
\bm{x}^{i}=\sqrt{\frac{2s}{dk}}
\begin{cases}
\begin{pmatrix}
\sin(\theta_{i})& \cos(\theta_{i})&\cdots &\sin(\tfrac{d}{2}\theta_{i})& \cos(\tfrac{d}{2}\theta_{i})
\end{pmatrix}^{\top} &\text{if $d$ is even}, \\[8pt]
\begin{pmatrix}
\tfrac{\sqrt2}{2}&
\sin(\theta_{i})& \cos(\theta_{i})&\cdots &
\sin(\lfloor d/2\rfloor \theta_{i})& \cos(\lfloor d/2\rfloor \theta_{i})
\end{pmatrix}^{\top} &\text{if $d$ is odd}.
\end{cases}
\]
Then
\[
\sum_{i=1}^{k}\bm{x}^{i}(\bm{x}^{i})^{\top}
=
\frac{s}{d}\bm I_{d} .
\]
\end{enumerate}
In either case, there exists  $s=s^*\in [0,k]$ such that $\bm{A}+\sum_{i=1}^{k} \bm{x}^{i}(\bm{x}^{i})^{\top}$ achieves the optimal eigenvalue allocation $\bm{t}+\bm{\beta}(s^*)$ and hence
attains the optimal value of problem~\eqref{eq:SpectralDesign}.
\end{proposition}
\begin{proof}
Case (i) is immediate since the vectors are orthogonal and
\[
\sum_{i=1}^{k} \bm{x}^{i}(\bm{x}^{i})^{\top}
=
\frac{s}{k}\sum_{i=1}^{k} \bm e_{i}\bm e_{i}^{\top} .
\]

We now prove case (ii).  We first note that for any
integer $m\in\mathbb Z$,
\[
\sum_{j=1}^{k} \mathrm{e}^{\mathrm{i} m \theta_{j}}
=
\begin{cases}
k,& m\equiv0 \pmod k,\\
0,&\text{otherwise}.
\end{cases}
\]
Indeed,
\[
\sum_{j=1}^{k} \mathrm{e}^{\mathrm{i} m \theta_{j}}
=
\sum_{j=0}^{k-1} \mathrm{e}^{\mathrm{i}2\pi m j/k}
=
\sum_{j=0}^{k-1}\omega^{j},
\qquad
\omega=\mathrm{e}^{\mathrm{i} 2\pi m/k},
\]
which equals $k$ if $\omega=1$ and otherwise vanishes by the geometric
series formula.

Taking real and imaginary parts yields
\[
\sum_{j=1}^{k} \cos(m \theta_{j})
=
\sum_{j=1}^{k} \sin(m \theta_{j})
=
0,
\qquad
m=1,\ldots,k-1.
\]

Using standard trigonometric identities,
\begin{align*}
\sum_{i=1}^{k} \sin(s\theta_{i})\cos(r\theta_{i})
&=\frac12\sum_{i=1}^{k}
\left[\sin((s+r)\theta_{i})+\sin((s-r)\theta_{i})\right],\\
\sum_{i=1}^{k} \sin(s\theta_{i})\sin(r\theta_{i})
&=\frac12\sum_{i=1}^{k}
\left[\cos((s-r)\theta_{i})-\cos((s+r)\theta_{i})\right],\\
\sum_{i=1}^{k} \cos(s\theta_{i})\cos(r\theta_{i})
&=\frac12\sum_{i=1}^{k}
\left[\cos((s-r)\theta_{i})+\cos((s+r)\theta_{i})\right].
\end{align*}

The Fourier identities imply that all cross terms vanish and
\[
\sum_{i=1}^{k} \sin(s\theta_{i})\sin(r\theta_{i})
=
\sum_{i=1}^{k} \cos(s\theta_{i})\cos(r\theta_{i})
=
\begin{cases}
\frac{k}{2}, & s=r,\\
0, & s\ne r .
\end{cases}
\]

Substituting these identities into the definition of $\bm{x}^{i}$
shows that the coordinates are orthogonal and have equal energy,
which yields
\[
\sum_{i=1}^{k} \bm{x}^{i}(\bm{x}^{i})^{\top}
=
\frac{s}{d}\bm I_{d} .
\]

It remains to connect these computations with optimality. By
\Cref{lem:general}, there exists $s^*\in[0,k]$ such that
$\bm{\beta}(s^*)$ solves the lower-bound problem
\eqref{eqn:relaxation-general}. 
Since $\bm A=\ell\bm I_{d}$, all coordinates of $\bm{t}$ are equal to
$\ell$. If $k\le d$, then the water level for budget $s$ is
$c(s)=\ell+s/k$, and \Cref{lem:simp_alpha} gives
$\bm{\beta}'(s)=(s/k,\ldots,s/k,0,\ldots,0)$, where $s/k$ is repeated
$k$ times. This is realized by case~\textup{(i)}. If $k\ge d+1$, then
the water level is $c(s)=\ell+s/d$, and
$\bm{\beta}(s)=(s/d,\ldots,s/d)$, which is realized by case~\textup{(ii)}.
Taking $s=s^*$, the constructed design attains the lower bound. Hence
it is optimal for problem~\eqref{eq:SpectralDesign}.
\end{proof}
The construction in \Cref{prop:closed_form_design} is closely related to
classical Fourier analysis \cite{stein2011fourier}.
When $k\ge d$, the vectors form a tight frame satisfying
\[
\sum_{i=1}^{k} \bm{x}^{i}(\bm{x}^{i})^{\top} = \frac{s}{d}\bm I_{d} ,
\]
which distributes the design uniformly across all directions.

\section{Numerical Illustrations}

The purpose of this section is twofold: first, to graphically
illustrate optimal spectral designs in two dimensions; and second, to
illustrate their use in DFO\@. The former is
presented in \Cref{subsect:optimal-design-illustrations}, and the
latter in \Cref{subsect:dfo}.
The numerical illustrations are performed for symmetric, convex, non-increasing $f$.

\subsection{Spectral Design Illustrations}
\label{subsect:optimal-design-illustrations}

This section presents graphical illustrations of optimal spectral designs in two dimensions,
that is, $d=2$, 
computed using the Python
package \texttt{spectraldesign}
\cite{Kleywegt2026}.

\begin{figure}[t]
\centering
\begin{subfigure}[t]{0.75\textwidth}
\centering
\includegraphics[width=0.32\textwidth]{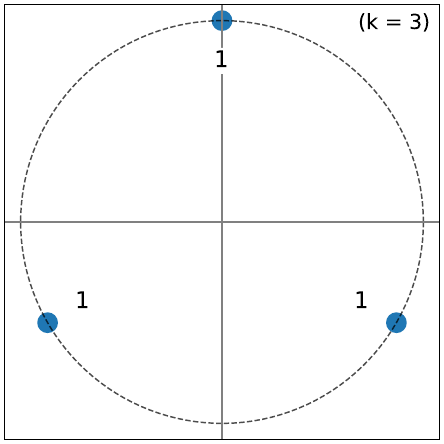}
\includegraphics[width=0.32\textwidth]{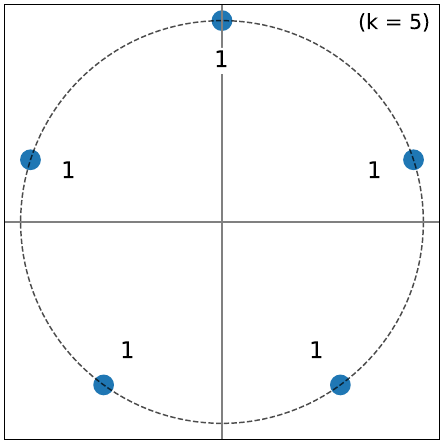}
\includegraphics[width=0.32\textwidth]{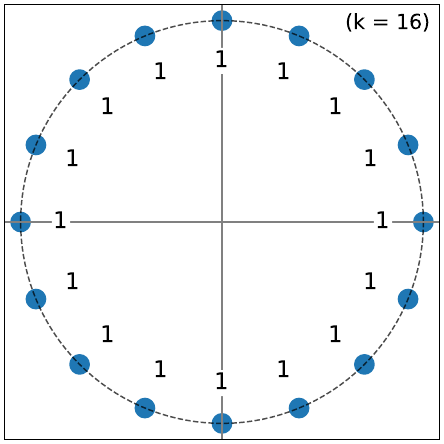}
\caption{Closed-form optimal designs.}
\end{subfigure}

\vspace{0.8em}

\begin{subfigure}[t]{0.75\textwidth}
\centering
\includegraphics[width=0.32\textwidth]{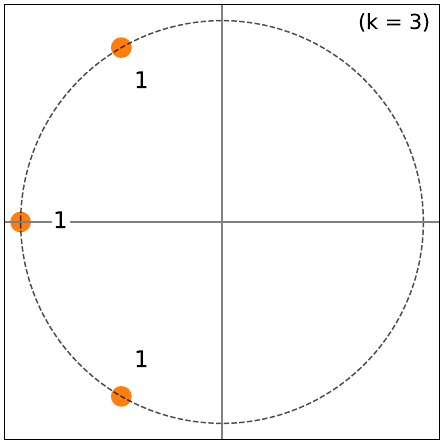}
\includegraphics[width=0.32\textwidth]{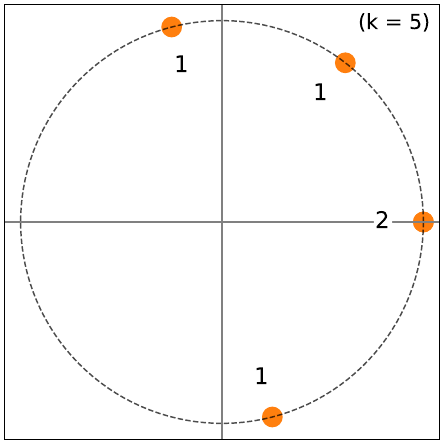}
\includegraphics[width=0.32\textwidth]{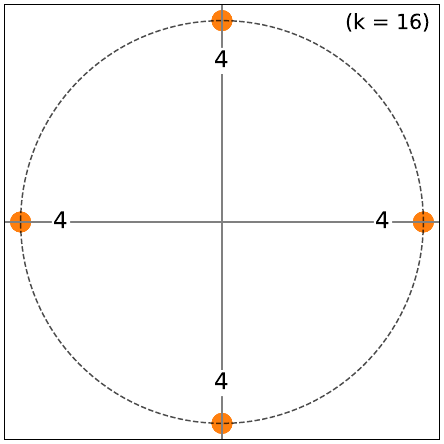}
\caption{Algorithmic optimal designs.}
\end{subfigure}
\caption{For three values of $k$ with zero prior information matrix
$\boldsymbol{A}=\boldsymbol{0}$, the optimal designs provided by
\Cref{prop:closed_form_design} \emph{(top)} and the optimal designs provided by the
spectral-design algorithm \emph{(bottom)}. The numbers next to the bullet
markers indicate the number of design vectors located at each site.}
\label{fig:closed-form-vs-algorithmic-solution}
\end{figure}

\Cref{fig:closed-form-vs-algorithmic-solution} compares two optimal designs for the case with no prior information, $\boldsymbol{A}=\boldsymbol{0}$. The top row shows the closed-form designs from \Cref{prop:closed_form_design}, whereas the bottom row shows the designs returned by the spectral-design algorithm. The numbers next to the markers indicate multiplicities, i.e., the number of design vectors placed at the same location. Although the closed-form and algorithmic constructions can produce different point configurations, they attain the same optimal value. This illustrates that the spectral design problem can have multiple optimal solutions: different sets of design vectors may generate updated information matrices with the same optimal eigenvalue profile.

\Cref{fig:one-prior-design} depicts optimal spectral designs for several values of $k$ when the prior matrix is rank one, $\boldsymbol{A}=\boldsymbol{x}_{0}\boldsymbol{x}_{0}^{\top}$, where $\boldsymbol{x}_{0}\in\mathbb{R}^{2}$. The prior vector is shown separately from the newly selected design vectors, and the multiplicities indicate how often the algorithm places design vectors at each selected point. The plots show that the prior direction affects both the locations and the multiplicities of the new design vectors. As $k$ increases, the algorithm does not simply add points uniformly; rather, it reallocates design effort to balance the spectrum of the updated information matrix $\boldsymbol{A}+\boldsymbol{X}\boldsymbol{X}^{\top}$. Thus, even in two dimensions with a rank-one prior, the optimal configuration depends in a nontrivial way on the interaction between the prior information and the number of additional design vectors.

\Cref{fig:two-prior-design} presents the corresponding designs when the prior matrix is generated by two prior design vectors, $\boldsymbol{A}=\boldsymbol{x}_{0}\boldsymbol{x}_{0}^{\top}+\boldsymbol{x}_{1}\boldsymbol{x}_{1}^{\top}$, where $\boldsymbol{x}_{0},\boldsymbol{x}_{1}\in\mathbb{R}^{2}$. Compared with \Cref{fig:one-prior-design}, the additional prior direction changes the spectrum and eigenspaces of the prior information matrix, which in turn changes the optimal placement of the new design vectors. Together, \Cref{fig:one-prior-design,fig:two-prior-design} demonstrate that optimal spectral designs are sensitive to both the eigenvalues and eigenvectors of the prior information matrix. These examples also illustrate why a constructive algorithm is useful: the optimal design is not always easy to infer from geometric intuition alone.

\begin{figure}[t]
\centering
\begin{subfigure}[t]{0.25\textwidth}
\centering
\includegraphics[width=0.95\textwidth]{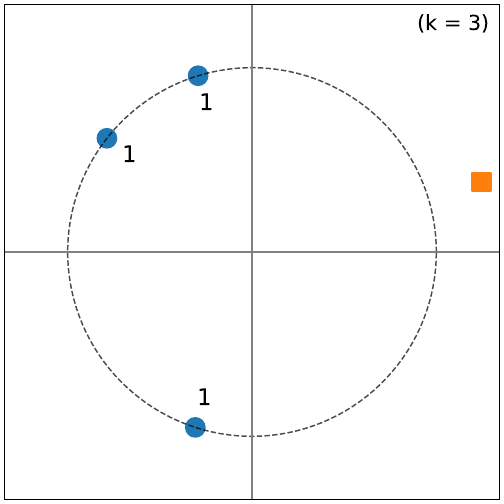}
\caption{$k=3$}
\end{subfigure}
\begin{subfigure}[t]{0.25\textwidth}
\centering
\includegraphics[width=0.95\textwidth]{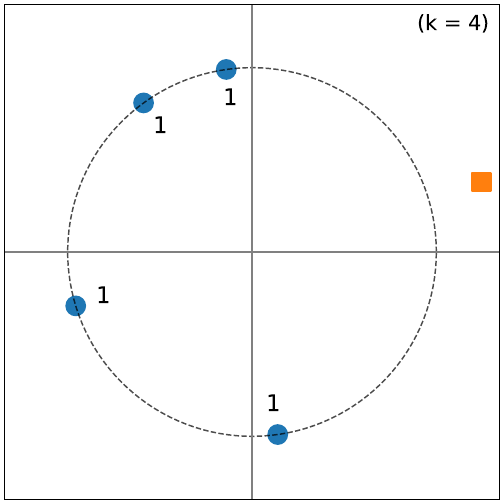}
\caption{$k=4$}
\end{subfigure}
\begin{subfigure}[t]{0.25\textwidth}
\centering
\includegraphics[width=0.95\textwidth]{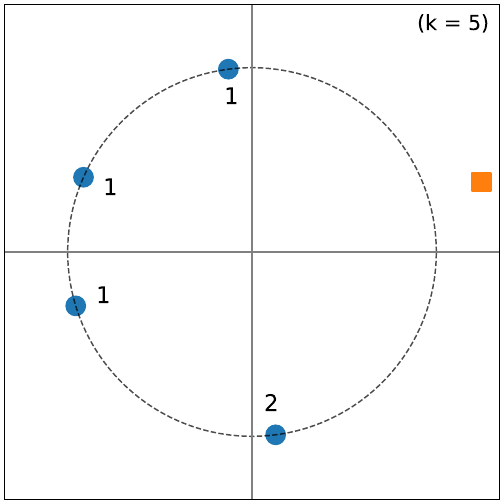}
\caption{$k=5$}
\end{subfigure}

\vspace{0.8em}

\begin{subfigure}[t]{0.25\textwidth}
\centering
\includegraphics[width=0.95\textwidth]{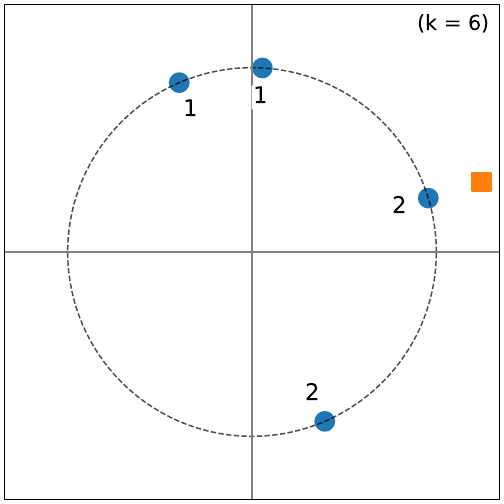}
\caption{$k=6$}
\end{subfigure}
\begin{subfigure}[t]{0.25\textwidth}
\centering
\includegraphics[width=0.95\textwidth]{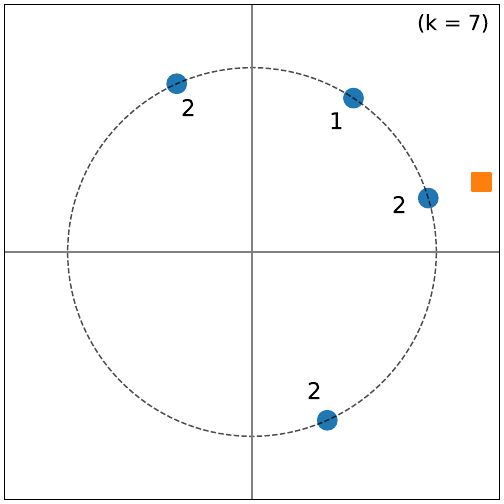}
\caption{$k=7$}
\end{subfigure}
\begin{subfigure}[t]{0.25\textwidth}
\centering
\includegraphics[width=0.95\textwidth]{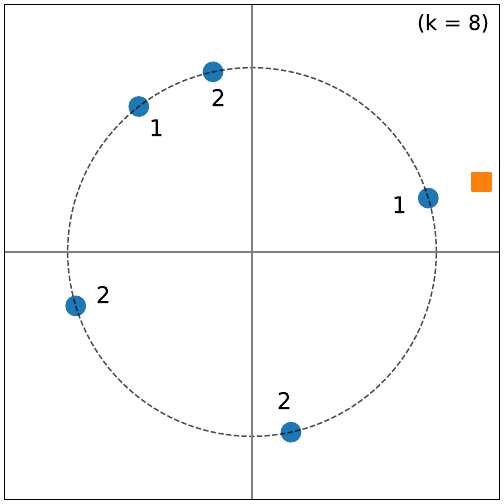}
\caption{$k=8$}
\end{subfigure}

\vspace{0.8em}

\begin{subfigure}[t]{0.25\textwidth}
\centering
\includegraphics[width=0.95\textwidth]{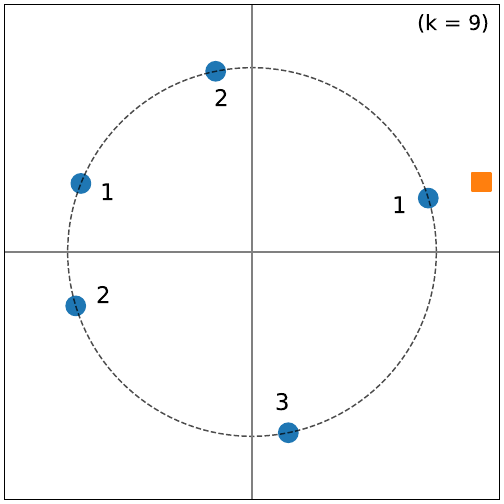}
\caption{$k=9$}
\end{subfigure}
\begin{subfigure}[t]{0.25\textwidth}
\centering
\includegraphics[width=0.95\textwidth]{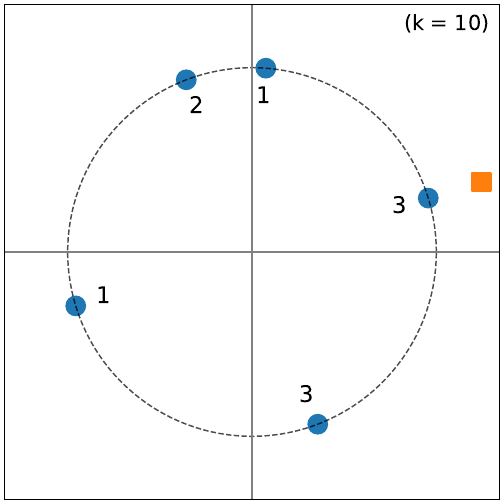}
\caption{$k=10$}
\end{subfigure}
\begin{subfigure}[t]{0.25\textwidth}
\centering
\includegraphics[width=0.95\textwidth]{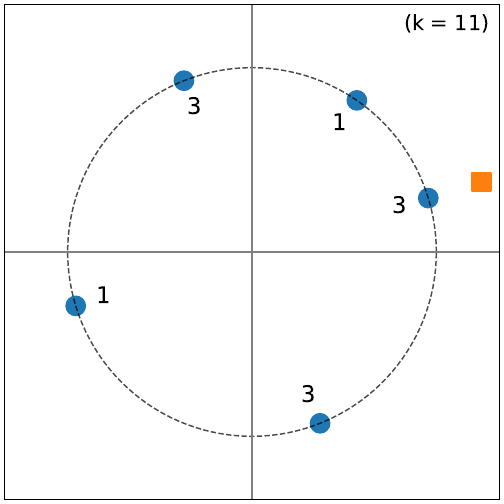}
\caption{$k=11$}
\end{subfigure}

\caption{For several $k$, optimal spectral designs with a single prior design vector,
that is, $\boldsymbol{A}=\boldsymbol{x}_{0}\boldsymbol{x}_{0}^{\top}$. The prior
vector $\boldsymbol{x}_{0}$ is shown as an orange square, and the optimal spectral designs computed by \texttt{spectraldesign} are
shown as blue dot markers. The numbers next to the bullet
markers indicate the number of design vectors located at each site.}
\label{fig:one-prior-design}
\end{figure}

\begin{figure}[t]
\centering
\begin{subfigure}[t]{0.25\textwidth}
\centering
\includegraphics[width=0.95\textwidth]{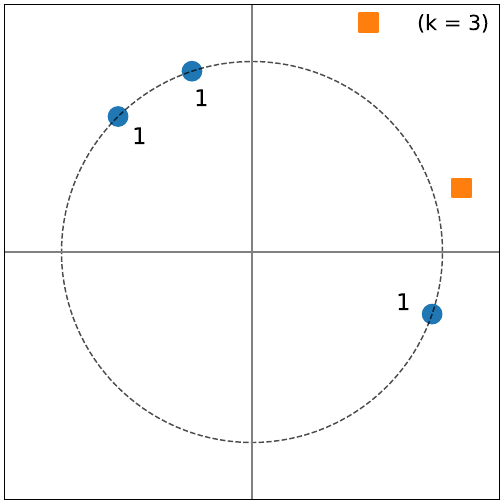}
\caption{$k=3$}
\end{subfigure}
\begin{subfigure}[t]{0.25\textwidth}
\centering
\includegraphics[width=0.95\textwidth]{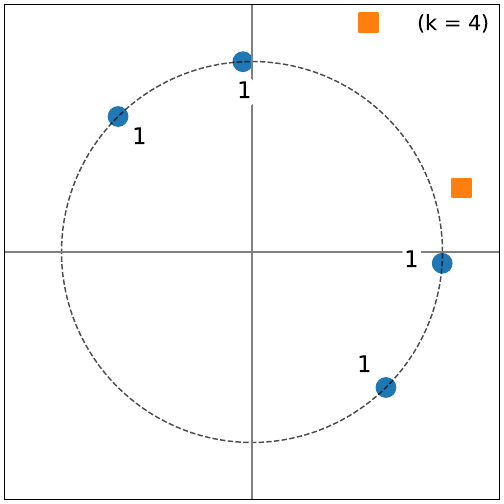}
\caption{$k=4$}
\end{subfigure}
\begin{subfigure}[t]{0.25\textwidth}
\centering
\includegraphics[width=0.95\textwidth]{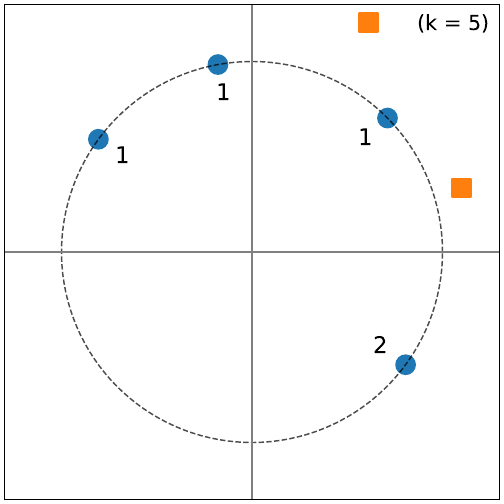}
\caption{$k=5$}
\end{subfigure}

\vspace{0.8em}

\begin{subfigure}[t]{0.25\textwidth}
\centering
\includegraphics[width=0.95\textwidth]{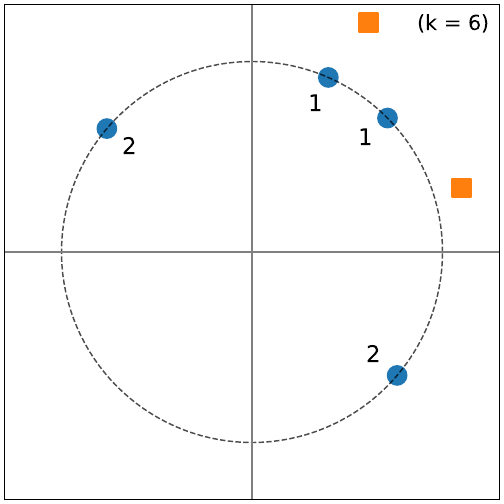}
\caption{$k=6$}
\end{subfigure}
\begin{subfigure}[t]{0.25\textwidth}
\centering
\includegraphics[width=0.95\textwidth]{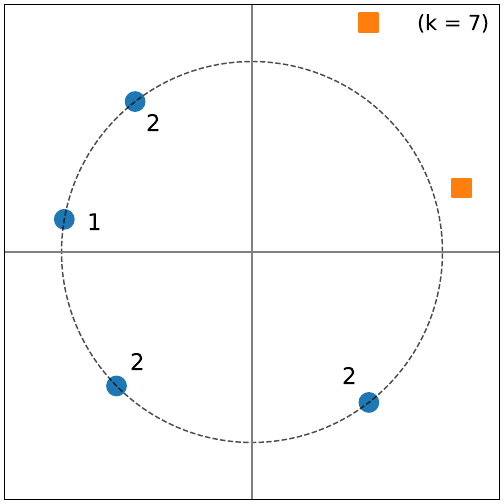}
\caption{$k=7$}
\end{subfigure}
\begin{subfigure}[t]{0.25\textwidth}
\centering
\includegraphics[width=0.95\textwidth]{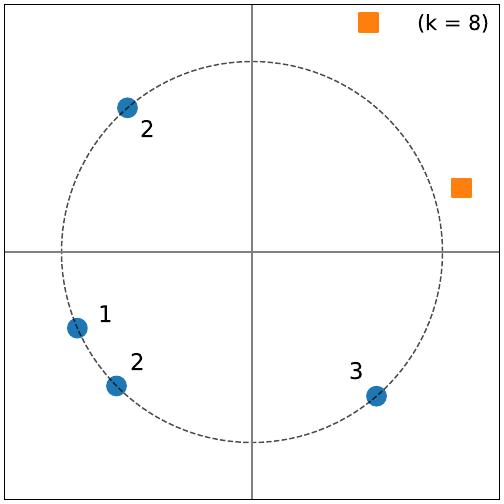}
\caption{$k=8$}
\end{subfigure}

\vspace{0.8em}

\begin{subfigure}[t]{0.25\textwidth}
\centering
\includegraphics[width=0.95\textwidth]{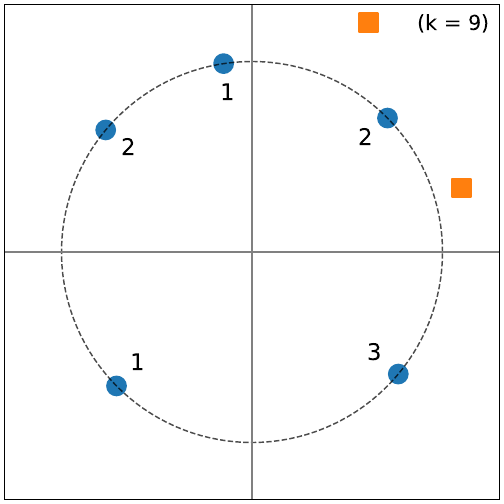}
\caption{$k=9$}
\end{subfigure}
\begin{subfigure}[t]{0.25\textwidth}
\centering
\includegraphics[width=0.95\textwidth]{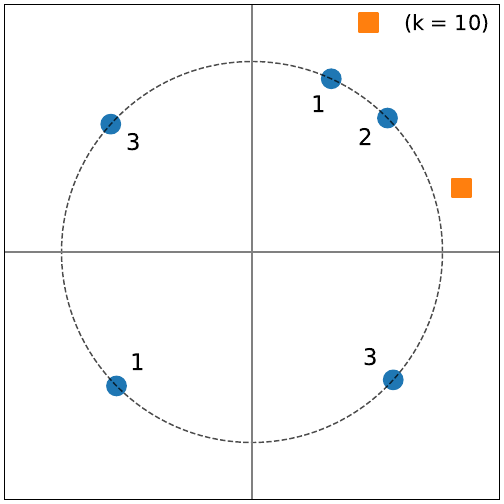}
\caption{$k=10$}
\end{subfigure}
\begin{subfigure}[t]{0.25\textwidth}
\centering
\includegraphics[width=0.95\textwidth]{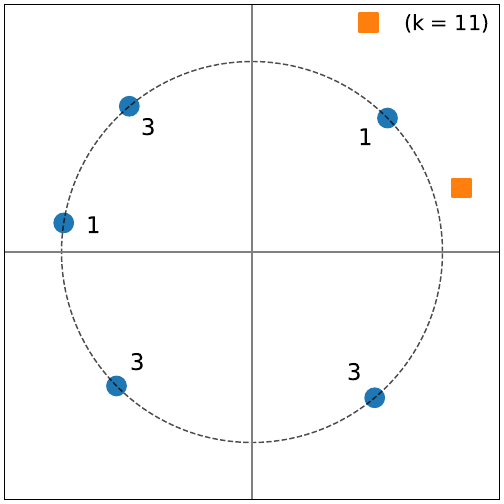}
\caption{$k=11$}
\end{subfigure}

\caption{For several $k$,  optimal spectral designs with two prior design vectors,
that is, $\boldsymbol{A}=\boldsymbol{x}_{0}\boldsymbol{x}_{0}^{\top}
+\boldsymbol{x}_{1}\boldsymbol{x}_{1}^{\top}$. The prior vectors $\boldsymbol{x}_{0}$ and $\boldsymbol{x}_{1}$ are shown as
orange squares, and the optimal spectral designs computed by \texttt{spectraldesign} are shown as blue dot
markers. The numbers next to the bullet
markers indicate the number of design vectors located at each site.}
\label{fig:two-prior-design}
\end{figure}

\subsection{DFO Illustrations}
\label{subsect:dfo}

We next illustrate the use of spectral designs within a first-order
model-based DFO procedure. These experiments are intended as a
proof-of-concept for using spectral designs inside a regression-based
DFO method. We do not provide a comprehensive benchmark against
state-of-the-art DFO solvers. We compare three solvers on smooth
unconstrained problems of the form~\eqref{eq:unconstrained-problem}.
Each solver observes only noisy function values, while the true
objective values are used solely for performance assessment. 
The computer code and simulation output are archived at
\cite{johannes_milz_2026_20209976}.

\paragraph{Test Problems.}
We use the $53$ smooth instances of the Mor\'e--Wild test set
\cite{More2009}, obtained from the \texttt{BenDFO} benchmark suite
\cite{POptUS2026}. The solvers do not observe $g$ directly. Instead,
each oracle call returns
\[
\tilde g(\bm{y})
=
g(\bm{y})+\xi,
\qquad
\xi \sim \mathcal{U}[-\sigma/2,\sigma/2],
\]
where the noise variables $\xi$ 
are drawn independently across oracle calls.
This additive-noise model is a standard way to stress-test DFO methods
under inexact function evaluations. Each problem is run
with $30$ independent random seeds, giving $1590$ test instances.

\paragraph{Algorithms and Parameter Choices.}
We compare three variants of the derivative-free bidirectional method
(DFBD) of \cite[Algorithm~4]{Khanh2025}. The first variant, denoted
``DFBD + forward differences,'' uses forward finite differences to
estimate gradients as in \cite{Khanh2025}. The second variant, denoted
``DFBD + spectral design,'' replaces the finite-difference estimate by
the regression-based estimate described in \Cref{sec:motivation}, with
new directions computed using the Python package
\texttt{spectraldesign}~\cite{Kleywegt2026}. The third variant, denoted
``DFBD + coordinate design,'' also uses the regression-based estimate,
but replaces the spectral-design directions by the first $k$ columns
of the identity matrix. DFBD is a useful baseline because it chooses the
finite-difference radius adaptively; in the regression-based variants,
we use this radius as the design-region radius $\delta$.

The DFBD parameters are $\eta=2$ and $L_{0}=1$, where $L_{0}$ is the
initial Lipschitz estimate. The bidirectional search index is capped at
$30$, and the noise level is set to $\sigma$.
In both regression-based variants, we form a reused-direction matrix
$\bm U$ whose columns are the normalized displacements of reusable
previous evaluations from the current incumbent, as described in
\Cref{sec:motivation}, with reuse radius $r=100$. The corresponding
prior information matrix used by the spectral-design procedure is
$\bm A=\bm U\bm U^{\top}$.
At each iteration, the number of new design
directions is computed as
\[
k =
\max\left\{
1,
\left\lfloor \frac{d}{2} \right\rfloor,
d-\operatorname{rank}(\bm U)
\right\}.
\]
The prior information matrix is updated whenever a regression-based gradient estimate is computed.

\paragraph{Data Profiles.}
We use data profiles to summarize performance over all test instances.
Fix one
such instance, and let $d$ be its dimension. For method $m$, let
$h_m(t)$ be the best true objective value found by method $m$ after
$t$ oracle calls. The maximum budget is $50(d+1)$ oracle calls.

Let $g(\bm{x}_0)$ denote the true objective value at the starting point
provided by the test set, and let $g^\star$ denote the best final true objective value obtained by any method on this same instance. Following
\cite{More2009}, we say that method $m$ solves the instance at
accuracy $\tau \in (0,1)$ within $t$ oracle calls if
\begin{align}
\label{eq:convergence-test}
h_m(t)
\le
\tau g(\bm{x}_0) + (1-\tau) g^\star .
\end{align}
Thus, the target value lies between the initial objective value and the
best final value found by the methods being compared.
For each normalized budget $\alpha$, the data profile reports the
fraction of test instances solved within $\alpha(d+1)$ oracle calls.
Runs that do not satisfy \eqref{eq:convergence-test} within the budget
are counted as unsolved. Higher curves, therefore, indicate better
performance.
All algorithmic decisions, including gradient estimation and
bidirectional-search decisions, use only noisy oracle values. The true
objective values are used only to compute $h_m(t)$ and the data
profiles.

\paragraph{Numerical Results.}

\Cref{fig:data-sigma-1e-2} reports data profiles for three noise levels, $\sigma\in\{10^{-1},10^{-2},10^{-3}\}$, and two accuracy levels, $\tau\in\{10^{-1},10^{-2}\}$. The horizontal axis is the number of oracle calls normalized by $d+1$, and the vertical axis is the fraction of problem--seed pairs satisfying the prescribed accuracy criterion. The plots compare three variants of DFBD: the finite-difference baseline, a regression-based variant using coordinate designs, and a regression-based variant using spectral designs.

Across these illustrative experiments, the two regression-based variants generally outperform the forward finite-difference baseline, especially when higher accuracy is required. The spectral-design variant is competitive with the coordinate-design variant across all six profiles and outperforms it in most cases. These experiments are intended as a proof of concept rather than a comprehensive benchmark against state-of-the-art DFO solvers. Nevertheless, the results suggest that spectral designs offer a promising approach to selecting sampling directions for regression-based DFO when function evaluations are noisy.

\begin{figure}[t]
\centering

\begin{subfigure}[t]{0.44\linewidth}
\centering
\includegraphics[width=\linewidth]{
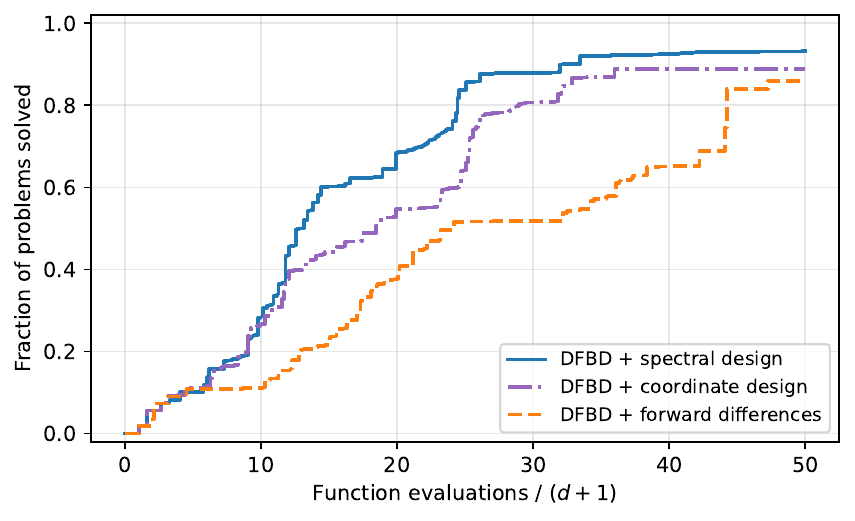}
\caption{$\sigma=10^{-1}$, and $\tau=10^{-1}$.}
\label{fig:data-sigma-1e-2-tau-1e-4}
\end{subfigure}
\hfill
\begin{subfigure}[t]{0.44\linewidth}
\centering
\includegraphics[width=\linewidth]{
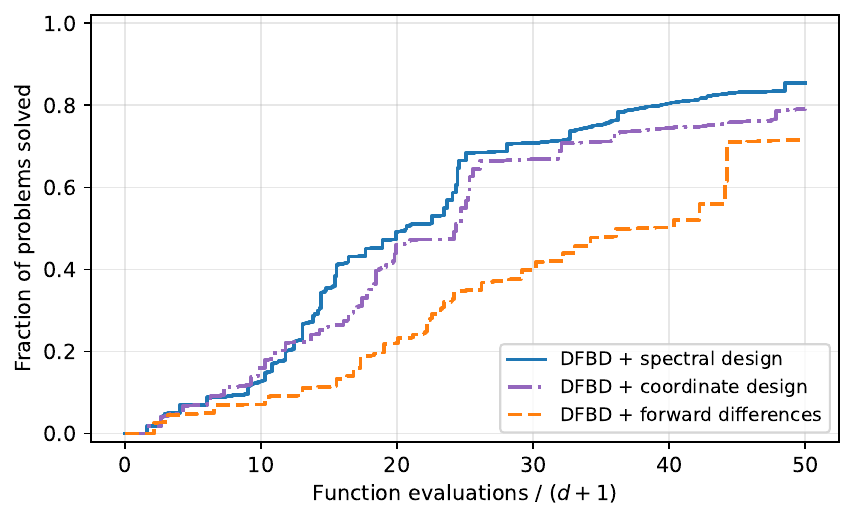}
\caption{$\sigma=10^{-1}$, and $\tau=10^{-2}$.}
\label{fig:data-sigma-1e-2-tau-1e-5}
\end{subfigure}

\vspace{0.5em}

\begin{subfigure}[t]{0.44\linewidth}
\centering
\includegraphics[width=\linewidth]{
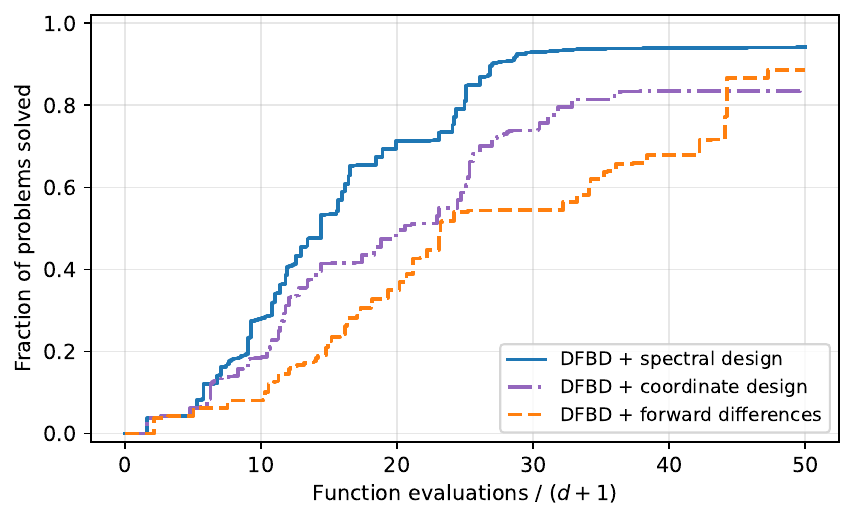}
\caption{$\sigma=10^{-2}$, and $\tau=10^{-1}$.}
\label{fig:data-sigma-1e-2-tau-1e-2}
\end{subfigure}
\hfill
\begin{subfigure}[t]{0.44\linewidth}
\centering
\includegraphics[width=\linewidth]{
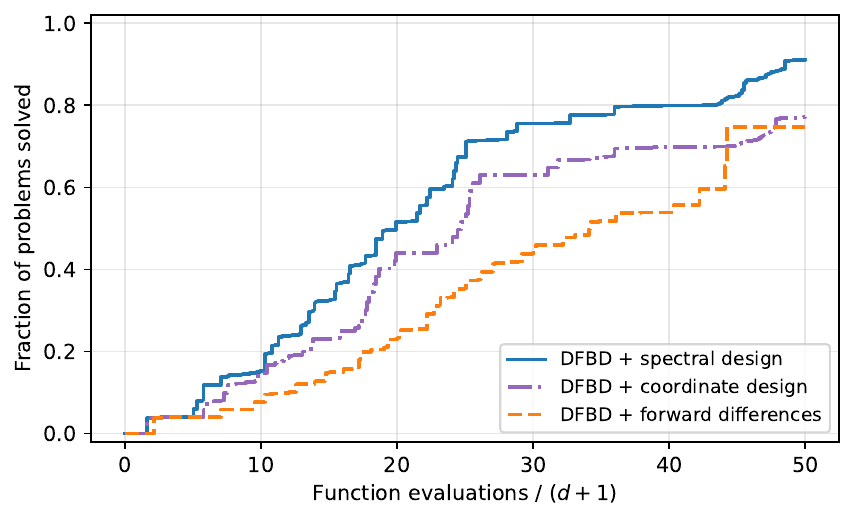}
\caption{$\sigma=10^{-2}$, and $\tau=10^{-2}$.}
\label{fig:data-sigma-1e-2-tau-1e-3}
\end{subfigure}

\vspace{0.5em}

\begin{subfigure}[t]{0.44\linewidth}
\centering
\includegraphics[width=\linewidth]{
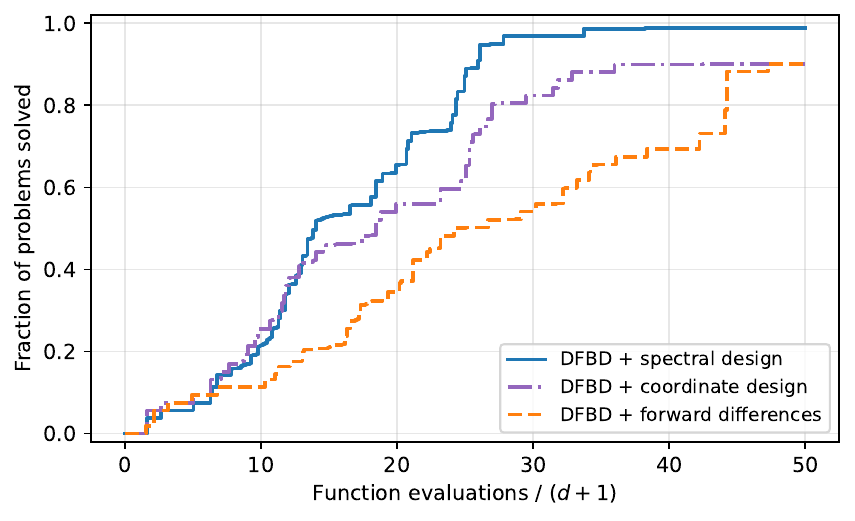}
\caption{$\sigma=10^{-3}$, and $\tau=10^{-1}$.}
\label{fig:data-sigma-1e-2-tau-1e-6}
\end{subfigure}
\hfill
\begin{subfigure}[t]{0.44\linewidth}
\centering
\includegraphics[width=\linewidth]{
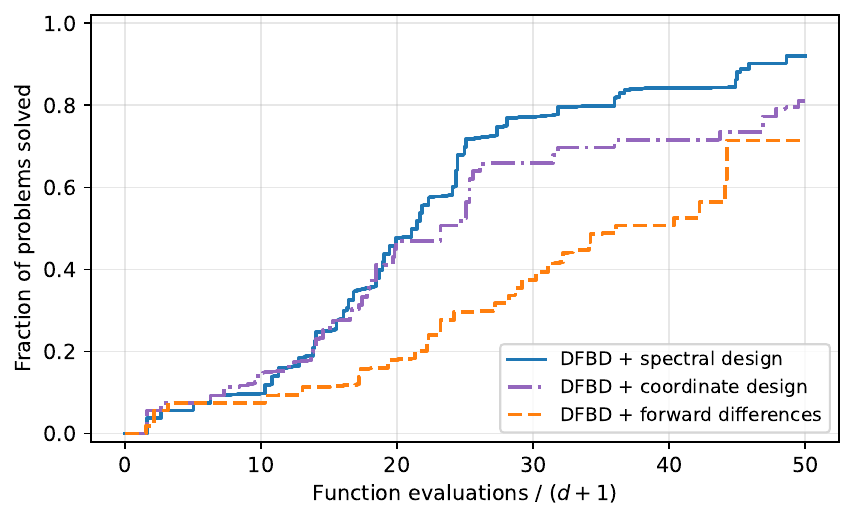}
\caption{$\sigma=10^{-3}$, and $\tau=10^{-2}$.}
\label{fig:data-sigma-1e-2-tau-1e-7}
\end{subfigure}

\caption{Data profiles  for several
noise levels $\sigma$ and accuracy levels $\tau$. The horizontal axis is
the normalized number of oracle calls, measured in multiples of $d+1$,
and the vertical axis is the fraction of problems satisfying
\eqref{eq:convergence-test}.}
\label{fig:data-sigma-1e-2}
\end{figure}

\section{Conclusion}

We studied a nonconvex spectral design problem motivated by optimal experimental design and DFO\@. Despite the nonconvexity of the matrix factorization inherent in the update structure, we showed that the problem can be solved optimally in polynomial time. A key insight is that, for non-increasing convex spectral criteria, optimality depends only on the eigenvalue structure of the prior matrix and the number of design vectors. Our analysis is based on a tight eigenvalue relaxation that admits a one-dimensional water-filling characterization. By combining majorization arguments with the Schur--Horn theorem, we constructed a feasible matrix factorization that attains the optimal relaxation value. The numerical illustrations suggest that the proposed designs can be
useful within DFO, particularly in noisy settings where stable gradient estimation is important. Future research directions include extending the framework to accommodate both first- and second-order objective function models, and studying whether a similar result persists in those settings.

\subsection*{Reproducibility of Numerical Illustrations}

The code and computational output used to generate the spectral-design
illustrations in \Cref{subsect:optimal-design-illustrations} are
archived at \url{https://doi.org/10.5281/zenodo.20347229}. The code and
computational output used for the DFO illustrations in
\Cref{subsect:dfo} are archived at \url{https://doi.org/10.5281/zenodo.20347568}.

\bibliography{main.bbl}

\end{document}